\documentclass{article}[12pt]

\usepackage{tikz}
\usepackage{amssymb}
\usepackage{amsthm}
\usetikzlibrary{shapes,snakes,arrows}
\usepackage{amsmath, xypic, graphicx}
\xyoption{all} \xyoption{arc}

\newcommand{\al}{\alpha}

\newcommand{\la}{\lambda}
\newcommand{\ty}{type}

\newtheorem{rem}{Remark}

\newtheorem{theorem}{Theorem}[section]
\newtheorem{proposition}{Proposition}[section]
\newtheorem{lemma}[theorem]{Lemma}
\newtheorem{definition}[theorem]{Definition}
\newtheorem{corollary}[theorem]{Corollary}
\newtheorem{example}[theorem]{Example}

\newtheorem{remark}[theorem]{Remark}

\tikzstyle{wv}=[circle,draw=black!90,fill=white!20,thick,inner sep=2pt,minimum width=5pt] 
\tikzstyle{bv}=[circle,draw=black!90,fill=black!100,thick,inner sep=2pt,minimum width=5pt] 
\tikzstyle{gv}=[circle,draw=black!70,fill=gray!100,thick,inner sep=2pt,minimum width=5pt]

\def\ie{{\it i.e.\ }}

\def\eg{{\it e.g.\ }}

\begin{document}
%\begin{frontmatter}

\title{Moments of normally distributed random matrices -- Bijective explicit evaluation }

\author{Ekaterina Vassilieva}%\ead{ekaterina.vassilieva@lix.polytechnique.fr}
%\address[lix]{Laboratoire d'Informatique de l'Ecole Polytechnique, 91128 Palaiseau Cedex, France}
\maketitle
\begin{abstract}
This paper is devoted to the distribution of the eigenvalues of $XUYU^{t}$ where $X$ and $Y$ are given symmetric matrices and $U$ is a random real valued square matrix of standard normal distribution. More specifically we look at its moments, i.e. the mathematical expectation of the trace of $(XUYU^{t})^n$ for arbitrary integer $n$.
 Hanlon, Stanley, Stembridge (1992) showed that this quantity can be expressed in terms of some generating series for the connection coefficients of the double cosets of the hyperoctahedral group with the eigenvalues of $X$ and $Y$ as indeterminate. We  provide an explicit evaluation of these series in terms of monomial symmetric functions. Our development relies on an interpretation of the connection coefficients in terms of locally orientable hypermaps and a new bijective construction between partitioned locally orientable hypermaps and some decorated forests. As a corollary we provide a simple explicit evaluation of the moments of $XUYU^{*}$ when $U$ is complex valued and $X$ and $Y$ are given hermitian matrices.
\end{abstract}
%\begin{keyword}matrix integrals, double coset algebra, connection coefficients, locally orientable hypermaps\end{keyword}
%\end{frontmatter}

\section{Introduction}
\subsection{Basic notations}
For any integer $n$, we note $[n]=\{1,\ldots,n\}$, $S_n$ the symmetric group on $n$ elements and $\la=(\la_1,\la_2,\ldots,\la_p) \vdash n$ an integer partition of $n$ with $\ell(\la)=p$ parts sorted in decreasing order. 
%Thus, $\la=(\la_1,\ldots, \la_k)$ where $\la_1 \geq \cdots \geq \la_k\geq 1$ and $\sum \la_i=n$. 
If $n_i(\la)$ is the number of parts of $\la$ that are equal to $i$ (by convention $n_0(\la)=0$), then we may write $\la $ as $[1^{n_1(\la)}\,2^{n_2(\la)}\ldots]$ and define $Aut_{\la}=\prod_i n_i(\la)!$ and $z_\lambda =\prod_i i^{n_i(\lambda)}n_i(\lambda)!$. We note $m_\lambda(x)$ and $p_\lambda(x)$ the monomial and power sum symmetric functions indexed by $\lambda$ on indeterminate $x$. For a $m\times m$ matrix $V$ we write $m_\la (V)$ and $p_\la(V)$ the value of these symmetric functions at the eigenvalues of $V$. Finally, for any real number $\al$ and $l$ non negative integers $k_1,\ldots,k_l$, we note the multinomial coefficients: 
\small
$$\binom{\al}{k_1,\ldots,k_l} = \frac{\al(\al-1)\ldots(\al-\sum_i k_i+1)}{k_1!k_2!\ldots k_l!}.$$
\normalsize
\subsection{Integer arrays}
\noindent Given $\la$ and $\mu$, two integer partitions of $n$ and a non negative integer $r$, we define the sets $M_{\la,\mu}^r$ of $4$-tuples ${\bf A}=(P,P',Q,Q')$ of bi-dimensional arrays of non negative integers such that:
$n_i(\mu) = \sum_{j \geq 0}Q_{ij} + Q'_{ij},  r =\sum_{i,j}j(Q_{ij} + Q'_{ij})$.
Additionally, there exist two indices $i_0$ and $j_0$ such that:
$
n_i(\la) =\delta_{i,i_0}+ \sum_{j \geq 0}P_{ij} + P'_{ij}, r =j_0 + \sum_{i,j}j(P_{ij} + P'_{ij})
$.
For such an array, we note: $p=|P| = \sum_{i,j\geq 0}P_{ij},  p'=\ell(\la)-p-1=|P'|, q=|Q|$, $q'=\ell(\mu)-q=|Q'|$,  ${\bf A!} = \prod_{i,j}P_{ij}!\,P'_{ij}!\,Q_{ij}!\,Q'_{ij}!$. We define as well $\mathcal{I}({\bf A})= i_0$ if $r=0$, otherwise:
\small
\begin{align*}
\mathcal{I}({\bf A}) = \binom{i_0}{j_0,j_0}\left [i_0-2j_0 + \frac{\sum_{i,j}{jQ'}(j_0(n-p)-ri_0)}{r^2}+\frac{\sum_{i,j}\left((n-q)j-ir\right)Q'\sum_{i,j}\left(i_0j-j_0(i-1)\right)P}{r^2(n-q-2r)}\right ]
\end{align*}
\normalsize
%Finally we note: $M_{\la,\mu} = \bigcup_{r \geq 0} M_{\la,\mu}^r$.

\subsection{Main results}
We look at the quantity $p_n(XUYU^{t})$ (i.e the trace of $(XUYU^{t})^n$) where $X$ and $Y$ are given $m\times m$ real symmetric matrices, $U$ is a random $m\times m$ real matrix whose entries are independent standard normal variables and $U^{t}$ is the transpose of $U$. The mathematical expectation of this quantity is of particular interest for statisticians (see \cite{OU}). We define: 
\begin{equation}
\label {real} P^{\mathbb{R}}_n(X,Y) = \mathcal{E}_{U}(p_n(XUYU^{t}))
\end{equation}
Similarly, for $X$ and $Y$ given $m\times m$ hermitian matrices and $U$, random complex matrix whose entries are independent standard normal variables, we define:
\begin{equation}
\label {complex} P^{\mathbb{C}}_n(X,Y) = \mathcal{E}_{U}(p_n(XUYU^{*}))
\end{equation}
Where $U^{*}$ is the conjugate transpose of $U$. In \cite{HSS} Hanlon, Stanley and Stembridge proved that both of these quantities can be expressed as a linear combination of the $p_\la(X)p_\mu(Y)$ for $\la, \mu \vdash n$. Furthermore, they show that the coefficients in these expansions are the {\bf connection coefficients} of two commutative subalgebras of the group algebra of the symmetric group, the {\bf class algebra} (in the case of complex matrices) and the {\bf double coset algebra} (in the case of real matrices). While these coefficients admit a very nice combinatorial interpretation, their explicit value is unknown in the general case. By interpreting these coefficients as the cardinalities of sets of {\bf locally orientable (partitioned) unicellular hypermaps} and by introducing  a new bijective construction between such hypermaps and decorated forests, we provide an explicit expansion of $P^{\mathbb{R}}_n(X,Y)$ in terms of the $m_\la(X)m_\mu(Y)$. Namely :

\begin{theorem}\label{thm:main}

Let $P^{\mathbb{R}}_n(X,Y)$ be defined as in Equation \ref{real}, we have :
\begin{multline*}  \label{eq:mainthm} 
P^{\mathbb{R}}_n(X,Y)= \sum_{\lambda,\mu \vdash n}m_{\lambda}(X)m_{\mu}(Y)Aut_{\lambda}Aut_{\mu}\times \\
\sum_{r\geq 0}\sum_{{\bf A}\in M^r_{\lambda,\mu}}\frac{\mathcal{I}({\bf A})}{{\bf A}!}\frac{r!^2(n-q-2r)!(n-1-p-2r)!}{2^{2r-p'-q'}(n-p-q-2r)!} 
\prod_{i,j}{\binom{i-1}{j,j}}^{(P+Q)_{i,j}}{\binom{i-1}{j,j-1}}^{(P'+Q')_{i,j}} 
\end{multline*}
\end{theorem}
 
Using a special case of our bijection we show the following result:
\begin{theorem}\label{thm:comp}
Let $P^{\mathbb{C}}_n(X,Y)$ be defined as in Equation \ref{complex}, we have :
\begin{equation*}
 \label{eq:complex} 
P^{\mathbb{C}}_n(X,Y)= n\sum_{\lambda,\mu \vdash n}m_{\lambda}(X)m_{\mu}(Y)\frac{(n-\ell(\la))!(n-\ell(\mu))!}{(n+1-\ell(\la)-\ell(\mu))!} 
\end{equation*}
\end{theorem}
Denote $I_l$ the $m\times m$ diagonal matrix whose first $l$ diagonal entries are equal to $1$ and the $m-l$ remaining ones equal $0$. In both theorems the special cases $X=I_l $ and $Y=I_m$ are of particular interest. Let  $Q^{\mathbb{R}}_n(l,m) = P^{\mathbb{R}}_n(I_l,I_m)$ (resp.  $Q^{\mathbb{C}}_n(l,m) = P^{\mathbb{C}}_n(I_l,I_m)$). As a corollary to our main results, we find:

\begin{corollary}\label{thm:cor}

Let $Q^{\mathbb{R}}_n(l,m)$ be defined as above, we have :
\begin{equation*}  \label{eq:corthm} 
\frac{1}{n!}Q^{\mathbb{R}}_n(l,m)=\sum_{r,p,q,p',q'}\binom{l}{p,p'}\binom{m}{q,q'}\binom{n+2r-1}{p+2r-1,q+2r-1}{\binom{n+2r-1}{r,r}}^{-1}2^{2r-p'-q'}\alpha_{r,p,q,p',q'},
\end{equation*}
where the summation indices check $p\geq1$, $q+q'\geq1$ and we have $\alpha_{0,p,q,p',q'} = \delta_{p'0}\delta_{q'0}$ and for $r>0$:
\begin{equation*}
\alpha_{r,p,q,p',q'} = \sum_{a,b}(-1)^{p'+q'-a-b}\left [\frac{p}{p+a}\left(1+\frac{aq}{(p+2r)(q+b)}\right)\right]\binom{-(p+a)/2}{r}\binom{-(q+b)/2}{r}\binom{p'}{a}\binom{q'}{b}
 \end{equation*}
\end{corollary}

\noindent We also have:
\begin{corollary}\label{thm:corcomp}

Let $Q^{\mathbb{C}}_n(l,m)$ be defined as above, then:
\begin{equation*}  \label{eq:corthm} 
\frac{1}{n!}Q^{\mathbb{C}}_n(l,m)=\sum_{p,q\geq 1}\binom{l}{p}\binom{m}{q}\binom{n-1}{p-1,q-1}
\end{equation*}
\end{corollary}

\subsection{Connection coefficients}
For $\la \vdash n$, let $\mathcal{C}_{\la}$ be the {\bf conjugacy class} in $S_n$ of permutations with cycle type $\la$. The cardinality of the conjugacy classes is given by $|C_\la| = n!/z_\la$. We denote by $\mathcal{C}_{\la\la}$ the set of permutations of $S_{2n}$ with cycle type $\la\la =(\la_1,\la_1,\la_2,\la_2,\ldots,\la_k,\la_k)$. We look at {\bf perfect pairings} of the set $[n]\cup [\widehat{n}]=\{1,\ldots n, \widehat{1},\ldots, \widehat{n}\}$ which we view as fixed point free involutions in $S_{2n}$. Note that for $f,g \in S_{2n}$, the disjoint cycles of the product $f\circ g$ have repeated lengths \ie $f\circ g \in \mathcal{C}_{\la\la}$ for some $\la \vdash n$. Additionally, $B_n$ is the {\bf hyperoctahedral group} (i.e the centralizer of $f_{\star}=(1\widehat{1})(2\widehat{2})\cdots(n\widehat{n})$). We note $K_\lambda$ the {\bf double coset} of $B_n$ in $S_{2n}$ consisting in all the permutations $\omega$ of $S_{2n}$ such that $f_\star \circ\omega\circ f_\star\circ\omega^{-1}$ belongs to $\mathcal{C}_{\la\la}$. We have $|B_n| = 2^nn!$ and $|K_\la| = |B_n|^2/(2^{\ell(\la)}z_\la)$.
By abuse of notation, let $C_\lambda$ (resp. $K_\lambda$) also represent the formal sum of its elements in the group algebra $\mathbb{C} S_{n}$ (resp. $\mathbb{C} S_{2n}$). Then, $\{C_\lambda, \lambda \vdash n\}$ (resp. $\{K_\lambda, \lambda \vdash n\}$) forms a basis of the class algebra (resp.  double coset algebra, i.e. the commutative subalgebra of $\mathbb{C} S_{2n}$ identified as the Hecke algebra of the Gelfand pair $(S_{2n},B_n)$).
For $\lambda$, $\mu$, $\nu \vdash n$, the {\bf connection coefficients} $c^\nu_{\lambda,\mu}$ and $b^\nu_{\lambda,\mu}$ can be defined formally by:
\begin{equation}
c^\nu_{\lambda,\mu} = [C_\nu]C_\lambda C_\mu, \;\;\;\;\; b^\nu_{\lambda,\mu} = [K_\nu]K_\lambda K_\mu
\end{equation}
From a combinatorial point of view $c^\nu_{\lambda\mu}$ is the number of ways to write a given permutation $\gamma_\nu$ of $C_\nu$ as the ordered product of two permutations $\alpha\circ\beta$ where $\alpha$ is in $C_\lambda$ and $\beta$ is in $C_\mu$. Similarly, $b^\nu_{\lambda\mu}$ counts the number of ordered factorizations of a given element in $K_\nu$ into two permutations of $K_\lambda$ and $K_\mu$.
We have the following results for $\nu = (n)$:
\begin{theorem}[\cite{HSS}]\label{thm:HSS}
\begin{equation} 
\label{eq:HSS} 
P^{\mathbb{R}}_n(X,Y)= \frac{1}{|B_n|}\sum_{\lambda,\mu \vdash n}b_{\lambda, \mu}^{n}p_{\lambda}(X)p_{\mu}(Y)
\end{equation}
\begin{equation}
\label{eq:HSS2}
P^{\mathbb{C}}_n(X,Y)= \sum_{\lambda,\mu \vdash n}c^n_{\lambda,\mu}p_{\lambda}(X)p_{\mu}(Y)
\end{equation}
\end{theorem}
We note $x^n_{p,q} = \sum_{\lambda,\mu \vdash n; \ell(\lambda) = p;\ell(\mu) = q}x_{\lambda, \mu}^{n}$ for $x = b$ or $c$. As an immediate corollary we have:
\begin{corollary}\label{cor:HSS}
\begin{equation}  
Q^{\mathbb{R}}_n(l,m)= \frac{1}{|B_n|}\sum_{p,q \geq 1}b^n_{p,q}l^pm^q
\end{equation}
\begin{equation}
\label{eq:cor2}
Q^{\mathbb{C}}_n(l,m)= \sum_{p,q \geq 1}c^n_{p,q}l^pm^q
\end{equation}
\end{corollary}
\subsection{Computation of connections coefficients}
Despite the attention the problem received and the elegant combinatorial interpretations of the coefficients $c^n_{\lambda,\mu}$ and $b_{\lambda, \mu}^{n}$, no closed formulas are known except for very special cases. Using an inductive argument B\'{e}dard and Goupil  \cite{BG} first found a formula for $c^n_{\lambda,\mu}$ in the case $\ell(\la)+\ell(\mu)=n+1$, which was later reproved by Goulden and Jackson \cite{GJ92} via a bijection with a set of ordered rooted bicolored trees. Later, using characters of the symmetric group and a combinatorial development, Goupil and Schaeffer \cite{GS} derived an expression for connection coefficients of arbitrary genus as a sum of positive terms (see Biane \cite{PB} for a succinct algebraic derivation; and Poulalhon and Schaeffer \cite{PS}, and Irving \cite{JI} for further generalizations). 
Closed form formulas can be found when considering the expansion of the generating series in the RHS of Theorem \ref{thm:HSS} (resp. of corollary \ref{cor:HSS}) in the basis of the $m_{\lambda}(X)m_{\mu}(Y)$ (resp. $\binom{l}{p}\binom{m}{q}$). Jackson (\cite{J}) computed an elegant expression for a generalized version of the RHS of Equation \ref{eq:cor2} whose specialization proves Theorem \ref{thm:comp} when combined with Theorem \ref{thm:HSS}. The proof is algebraic and relies on the theory of the characters of the symmetric group. Schaeffer and Vassilieva in \cite{SV}, Vassilieva in \cite{V} and Morales and Vassilieva in \cite{MV} and \cite{MV2} provided the first purely bijective computations of the generating series in the RHS of (\ref{eq:HSS2}) and (\ref{eq:cor2}).\\
\indent Known results about the coefficients $b_{\lambda, \mu}^{n}$ are much more limited. As shown in \cite{HSS,GJ1}, the generating function in the RHS of equation of Equation \ref{eq:HSS} can be expanded in the basis of {\em zonal polynomials} with simple coefficients. The expression of zonal polynomials in terms of monomial symmetric function is however non trivial and unknown in the general case. In \cite{GJ3}, Goulden and Jackson conjectures that the coefficients $b_{\lambda, \mu}^{\nu}$ can be expressed as a counting series for hypermaps in locally orientable surfaces with respect to some statistics and proved the conjecture for $\la = [1^n]$ and $[1^{n-1}2^1]$.\\
\indent In this paper, we provide the first explicit monomial expansion of the RHS of Equation \ref{eq:HSS} thanks to a new bijection for locally orientable hypermaps. As the method is purely bijective, it provides a combinatorial interpretation of the coefficients in the monomial expansion and allows simple alternative combinatorial computations of some of these coefficients.   When specialized to the case of orientable hypermaps the proposed bijection simplifies considerably and becomes equivalent to the bijection proposed in \cite{MV}. This special case provides the monomial expansion of the RHS of Equation \ref{eq:HSS2} and proves Theorem \ref{thm:comp}. Using the proper parameters, the bijection and its special case prove the formulas of Theorems \ref{thm:cor} and \ref{thm:corcomp}.

\section{Combinatorial formulation}
\subsection{Unicellular locally orientable hypermaps}
From a topological point of view, a {\bf locally orientable hypermap} of n edges can be defined as a connected bipartite graph with black and white vertices. Each edge is composed of two half edges both connecting the two incident vertices.
This graph is embedded in a locally orientable surface such that if we cut the graph from the surface, the remaining part consists of connected components called  faces or cells, each homeomorphic to an open disk. The map can be represented as a ribbon graph on the plane keeping the incidence order of the edges around each vertex. In such a representation,  two half edges can be parallel or cross in the middle. A  crossing (or a twist) of two half edges indicates a change of orientation in the map and that the map is embedded in a non orientable surface (projective plane, Klein bottle,...).  
%Graphical examples of locally orientable hypermaps can be found in \cite{FS}.
We say a hypermap is {\bf rooted} if it has a distinguished half edge. In \cite{GJ1}, it was shown that rooted hypermaps admit a natural formal description involving triples of perfect pairings $(f_1, f_2, f_3)$ on the set of half edges  where:
\begin{itemize}
\item $f_3$ associates half edges of the same edge,
\item $f_1$ associates immediately successive (i.e. with no other half edges in between) half edges moving around the white vertices, and 
\item $f_2$ associates immediately successive half edges moving around the black vertices.
\end{itemize}

Formally we label each half edge with an element in $[n]\cup [\widehat{n}]=\{1,\ldots,n,\widehat{1},\ldots,\widehat{n}\}$, labeling the rooted half edge by $1$. We then define $(f_1, f_2, f_3)$ as perfect pairings on this set.
Combining the three pairings gives the fundamental characteristics of the hypermap since:
\begin{itemize}
\item The cycles of $f_3\circ f_1$ give the succession of edges around the white vertices. If $f_3\circ f_1 \in \mathcal{C}_{\lambda\lambda}$ then the degree distribution of the white vertices is $\lambda$ (counting only once each pair of half edges belonging to the same edge),
\item The cycles of $f_3\circ f_2$ give the succession of edges around the black vertices. If $f_3\circ f_2 \in \mathcal{C}_{\mu\mu}$ then the degree distribution of the black vertices is $\mu$ (counting only once each pair of half edges belonging to the same edge),
\item The cycles of $f_1\circ f_2$ encode the faces of the map. If $f_1\circ f_2 \in \mathcal{C}_{\nu\nu}$ then the degree distribution of the faces is $\nu$
\end{itemize}
In what follows, we consider the number  $L_{\lambda, \mu}^{n}$ of rooted {\bf unicellular}, or one-face, locally orientable hypermaps with face distribution $\nu=(n)=n^1$, white vertex distribution $\lambda$, and black vertex distribution $\mu$.

Let $f_1$ be the pairing $(1\,\widehat{n})(2\,\widehat{1})(3\,\widehat{2})\ldots(n\,n\widehat{-}1)$ and $f_2 =f_{\star}= (1\,\widehat{1})(2\,\widehat{2})\ldots(n\,\widehat{n})$. We have $f_1\circ f_2 = (123\ldots n)(\widehat{n}n\widehat{-}1\, n\widehat{-}2\ldots\widehat{1}) \in \mathcal{C}_{(n)(n)}$. Then one can show that
\begin{equation}
L_{\lambda, \mu}^{n} =\,\, \mid \{ f_3 \mbox{ pairings in } S_{2n}([n]\cup [\widehat{n}]) ; f_3\circ f_1 \in \mathcal{C}_{\lambda\lambda}, f_3\circ f_2 \in \mathcal{C}_{\mu\mu} \} \mid.
\end{equation}
Moreover the following relation between $L^n_{\la,\mu}$ and $b^n_{\la,\mu}$ holds \cite[Cor 2.3]{GJ1}:
\begin{equation}
\label{eq:lb}
L_{\lambda, \mu}^{n} = \frac{1}{2^nn!}b_{\lambda, \mu}^{n}
\end{equation}
Thus we can encode the connection coefficients as numbers of locally orientable hypermaps.\\
%Hypermaps are non orientable when there is a least one hat number paired to another hat number through $f_3$ (equivalently if a non-hat number is paired to a non-hat number through $f_3$). If the hypermaps is orientable we go through each edge of the map in both directions and there are no changes of direction during the map traversal.  
We can refine the definition of $L_{\lambda, \mu}^{n}$ using the non negative number $r$ of hat/hat (equivalently non-hat/non-hat) pairs in $f_3$:  
\begin{equation}
L_{\lambda, \mu, r}^{n} =\,\, \mid \{ f_3 \mbox{ pairings in } S_{2n}([n]\cup [\widehat{n}]) ; f_3\circ f_1 \in \mathcal{C}_{\lambda\lambda}, f_3\circ f_2 \in \mathcal{C}_{\mu\mu},\mid f_3([\widehat{n}])\cap[\widehat{n}]\mid = r \} \mid.
\end{equation}
Obviously $L_{\lambda, \mu}^{n} = \sum_{r \geq 0}L_{\lambda, \mu, r}^{n}$. The following result holds \cite[Prop 4.1]{GJ3}:
\begin{equation}
\label{eq: or}
L_{\lambda, \mu, 0}^{n} = c_{\lambda, \mu}^{n}
\end{equation}
%Let $f_1$ be the pairing $f_{\star}=(1\,\widehat{n})(2\,\widehat{1})(3\,\widehat{2})\ldots(n\,n\widehat{-}1)$ and $f_2 = (1\,\widehat{1})(2\,\widehat{2})\ldots(n\,\widehat{n})$. We have $f_1\circ f_2 = (123\ldots n)(\widehat{n}n\widehat{-}1\, n\widehat{-}2\ldots\widehat{1}) \in \mathcal{C}_{(n),(n)}$. 
%Using  \cite{GJ1}, one can easily show the following relation:
%\begin{equation}
%\label{eq:lb}
%L_{\lambda, \mu}^{n} = \frac{1}{2^nn!}b_{\lambda, \mu}^{(n)}
%\end{equation}
%Furthermore, following Lemma $3.2$ of \cite{HSS}, we have:
%\begin{equation}
%L_{\lambda, \mu}^{n} = \mid \{ f_3 \mbox{ pairing such that } f_3\circ f_1 \in \mathcal{C}_{\lambda,\lambda} \mbox{ and } f_3\circ f_2 \in \mathcal{C}_{\mu,\mu} \} \mid
%\end{equation}

\begin{example}
Figure \ref{fig:example} depicts a locally orientable unicellular hypermap in $L_{\lambda, \mu, r}^{n}$ with $\lambda=[1^12^23^14^1]$, $\mu = [3^14^15^1]$ and $r=3$ (at this stage we disregard the geometric shapes around the vertices).
\end{example}

\begin{figure}[htbp]
  \begin{center}
    \includegraphics[width=0.33\textwidth]{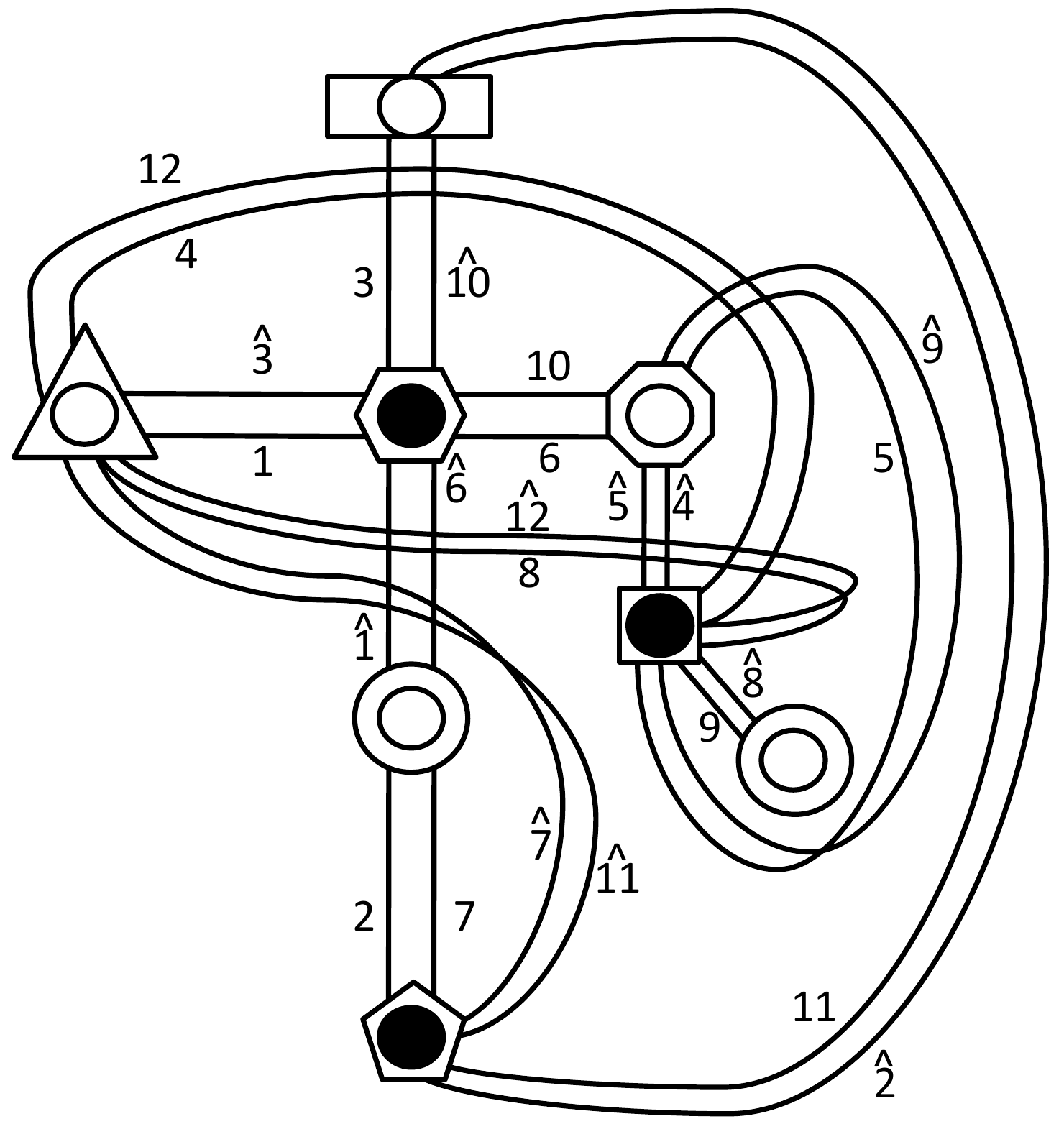}
    \caption{A unicellular locally orientable hypermap}
    \label{fig:example}
  \end{center}
\end{figure}
\subsection{Partitioned locally orientable hypermaps}
We consider locally orientable hypermaps where we partition the set of white vertices (resp. black). In terms of the pairings, this means we ``color'' the cycles of $f_3\circ f_1$ (resp. $f_3\circ f_2$) allowing repeated colors but imposing that the two cycles corresponding to each white (resp. black) vertex have the same color. The following definition in terms of set partitions of $[n]\cup [\widehat{n}]$ makes this more precise.
  
\begin{definition}[Locally orientable partitioned hypermaps]
We consider the set $\mathcal{LP}_{\lambda,\mu}^{n}$ of triples  $(f_3,\pi_1,\pi_2)$ where $f_3$ is a pairing on $[n]\cup [\widehat{n}]$, $\pi_1$ and $\pi_2$ are set partitions on  $[n]\cup [\widehat{n}]$ with blocks of even size and of respective types $2\la$ and $2\mu$ (or {\em half types} $\lambda$ and $\mu$)  with the constraint that $\pi_i$ $(i=1,2)$ is stable by $f_i$ and $f_3$.  Any such triple is called a {\bf locally orientable partitioned hypermap} of type $(\lambda,\mu)$. In addition, let $LP_{\lambda, \mu}^{n} = \mid \mathcal{LP}_{\lambda,\mu}^{n} \mid$. Finally, we note $\mathcal{LP}_{\lambda,\mu,r}^{n}$ and $LP_{\lambda, \mu, r}^{n}$ the set and number of such partitioned hypermaps with $\mid f_3([\widehat{n}])\cap[\widehat{n}]\mid = r$. 
\end{definition}

\begin{remark}
The analogous notion of partitioned or colored map is common in the study of orientable maps (\eg see \cite{L},\cite{GN}). Recently Bernardi in \cite{B} extended the approach in \cite{L} to find a bijection between locally orientable partitioned maps and orientable partitioned maps with a distinguished {\em planar submap}. As far as we know \cite[Sect. 7]{B} this technique does not extend to locally orientable hypermaps.
\end{remark}

\begin{lemma} The number of hat numbers in a block is equal to the number of non hat numbers\end{lemma}
\begin{proof} If a non hat number $i$ belongs to block $\pi_1^k$ then $f_1(i) = i\widehat{-}1$ also belongs to $\pi_1^k$. The same argument applies to blocks of $\pi_2$ with $f_2(i)=\widehat{i}$.\end{proof}
\begin{example}
As an example, the locally orientable hypermap on Figure \ref{fig:example} is partitioned into the blocks: 
\begin{eqnarray}
\nonumber \pi_1 &=& \{\{\widehat{12},1,\widehat{3},4,\widehat{7},8,\widehat{11},12\};\{\widehat{1},2,\widehat{6},7,\widehat{8},9\};\{\widehat{2},3,\widehat{10},11\};\{\widehat{4},5,\widehat{5},6,\widehat{9},10\}\}\\
\nonumber \pi_2 &=& \{\{1,\widehat{1},3,\widehat{3},6,\widehat{6},10,\widehat{10}\};\phantom{\{}\{2,\widehat{2},7,\widehat{7},11,\widehat{11}\};\phantom{\{}\{4,\widehat{4},5,\widehat{5},8,\widehat{8},9,\widehat{9},12,\widehat{12}\}\}
\end{eqnarray}
(blocks are depicted by the geometric shapes around the vertices, all the vertices belonging to a block have the same shape).
\end{example}

Let $\overline{R}_{\lambda,\mu}$ be the number of unordered partitions $\pi=\{\pi^{1}, \ldots, \pi^{p}\}$ of the set $[\ell(\lambda)]$ such that $\mu_j = \sum_{i\in \pi^j} \lambda_i$ for $1\leq j \leq \ell(\mu)$. Then for the monomial and power symmetric functions $m_{\lambda}$ and $p_{\lambda}$ we have $p_{\lambda} = \sum_{\mu \succeq \lambda} Aut_{\mu} \overline{R}_{\lambda,\mu} m_{\mu}$ \cite[Prop.7.7.1]{EC2}. We use this to obtain a relation between $L_{\lambda,\mu}^n$ and $LP_{\lambda,\mu}^n$.
 
\begin{proposition} \label{prop1} For partitions $\rho, \epsilon \vdash n$ and $r\geq 0$ we have  $LP_{\nu,\rho,r}^n =  \sum_{\lambda,\mu} \overline{R}_{\lambda\nu}\overline{R}_{\mu\rho}L_{\lambda,\mu,r}^n$, where  $\lambda$ and $\mu$ are refinements of $\nu$ and $\rho$ respectively.
\end{proposition}

\begin{proof} 
Let $(f_3,\pi_1,\pi_2) \in \mathcal{LP}_{\nu,\rho,r}^n$. If $f_3\circ f_1 \in \mathcal{C}_{\lambda\lambda}$ and $f_3\circ f_2 \in \mathcal{C}_{\mu\mu}$  then by definition of the set partitions we have that $\lambda$ and $\mu$ are refinements of  $\ty(\pi_1)=\nu$ and $\ty(\pi_2)=\rho$ respectively. Thus, we can classify the elements of $\mathcal{LP}_{\nu,\rho,r}^n$ by the cycle types of $f_3\circ f_1$ and $f_3\circ f_2$.  \ie  $\mathcal{LP}_{\nu,\rho,r}^n=\bigcup_{\lambda,\mu} \mathcal{LP}_{\nu,\rho,r}^n(\lambda,\mu)$, where
$$
\mathcal{LP}_{\nu,\rho,r}(\lambda,\mu) = \{ (f_3,\pi_1,\pi_2) \in \mathcal{LP}^n_{\nu,\rho,r} ~|~ (f_3\circ f_1,f_3\circ f_2) \in \mathcal{C}_{\lambda\lambda} \times \mathcal{C}_{\mu\mu}\}.
$$
If $LP_{\mu\rho,r}^n(\lambda,\mu) = |\mathcal{LP}_{\mu\rho,r}^n(\lambda,\mu)|$ then it is easy to see that $LP_{\mu,\rho,r}^n(\lambda,\mu)= \overline{R}_{\lambda\nu}\overline{R}_{\mu\rho} L_{\lambda\mu,r}^{n}$. 
\end{proof}

The change of basis between $p_{\lambda}$ and $m_{\lambda}$ immediately relates the generating series for $L^n_{\lambda,\mu,r}$ and the generating series for $LP^n_{\lambda,\mu,r}$ in  monomial symmetric functions: 
\begin{equation} \label{lem:lp}
\sum_{\lambda,\mu \vdash n}L_{\lambda, \mu,r}^{n}p_{\lambda}({\bf x})p_{\mu}({\bf y}) = \sum_{\lambda,\mu \vdash n}Aut_{\lambda}Aut_{\mu}LP_{\lambda, \mu,r}^{n}m_{\lambda}({\bf x})m_{\mu}({\bf y})
\end{equation}
Summing over $r$ gives:
\begin{equation}
\sum_{\lambda,\mu \vdash n}L_{\lambda, \mu}^{n}p_{\lambda}({\bf x})p_{\mu}({\bf y}) = \sum_{\lambda,\mu \vdash n}Aut_{\lambda}Aut_{\mu}LP_{\lambda, \mu}^{n}m_{\lambda}({\bf x})m_{\mu}({\bf y})
\end{equation}
For $l$ and $p$ non negative integers, we note $(l)_p = l(l-1)\ldots(l-p+1)$. We have  $m_\la(I_l) = (l)_{\ell(\la)}/Aut_\la$ and:
\begin{equation}
\label{eq:nodeg}
\sum_{p,q}L_{p, q, r}^{n}l^pm^q =\sum_{p,q}LP_{p, q, r}^{n}(l)_p(m)_q,
\end{equation}
where $LP_{p,q,r}^n = \sum_{\lambda,\mu \vdash n; \ell(\lambda) = p;\ell(\mu) = q}LP_{\lambda, \mu, r}^{n}$ (a similar definition applies to $L_{p,q,r}^n$).

\begin{definition}
Let  $\mathcal{LP}({\bf A})$ be the set of cardinality $LP({\bf A})$ of partitioned locally orientable hypermaps with $n$ edges where ${\bf A}= (P,P',Q,Q')$ are bidimensional arrays such that for $i,j \geq 0$:
\begin{itemize}
\item $P_{ij}$ (resp. $P'_{ij}$)  is the number of blocks of $\pi_1$ of half size $i$ that do not contain $1$ and such that: 
\begin{itemize}
\item[(i)]   its maximum {\bf non-hat} number is paired to a {\bf hat} (resp.  {\bf non-hat}) number by $f_3$  
\item[(ii)] the block contains $j$ pairs $\{t,f_3(t)\}$ where both $t$ and $f_3(t)$ are {\bf non-hat} numbers.
\end{itemize}

\item $Q_{ij}$ (resp. $Q'_{ij}$) is the number of blocks of $\pi_2$ of half size $i$ such that:
\begin{itemize}
\item[(i)] the maximum {\bf hat} number of the block is paired to a {\bf non-hat} (resp. {\bf hat}) number by $f_3$,
\item[(ii)] the block contains $j$ pairs $\{t,f_3(t)\}$ where both $t$ and $f_3(t)$ are {\bf hat} numbers.
\end{itemize}
\end{itemize}
\end{definition}
As a direct consequence we get:

\begin{equation}\label{eq:lpa}{LP}_{\lambda,\mu,r}^{n} = \sum_{{\bf A} \in M_{\lambda,\mu}^r}LP({\bf A})\end{equation}
\begin{example}
The partitioned hypermap on Figure \ref{fig:example} belongs to $\mathcal{LP}({\bf A})$ for $P = E_{3,1}+ E_{2,0}$, $P' = E_{3,1}$, $Q= E_{5,1}+E_{4,1}$, $Q' = E_{3,1}$ where $E_{t,u}$ is the  elementary array with entry $1$ at position $(t,u)$ and $0$ elsewhere.
\end{example}

\subsection{Permuted forests and reformulation of the main theorem}
We show that partitioned locally orientable hypermaps admit a nice bijective interpretation in terms of some recursive forests defined as follows:
\begin{definition}[Rooted bicolored forests of degree {\bf A}]\label{def:forest} In what follows we consider the set $\mathcal{F}({\bf A})$ of permuted rooted forests composed of:
\begin{itemize}
\item a bicolored identified ordered {\bf seed tree} with a white root vertex,
\item other bicolored ordered trees, called {\bf non-seed trees}  with either a white or a black root vertex,
\item each vertex of the forest has three kind of ordered descendants: {\bf tree-edges} (connecting a white and a black vertex), {\bf thorns} (half edges connected to only one vertex) and {\bf loops} connecting a vertex to itself. The two {\em extremities} of the loop are part of the ordered set of descendants of the incident vertex and therefore the loop can be intersected by thorns, edges and other loops as well.
\end{itemize}
The forests in $\mathcal{F}({\bf A})$ also have the following properties:
\begin{itemize}
\item the root vertices of the non-seed trees have at least one descending loop with one extremity being the rightmost descendant of the considered vertex,
\item the total number of thorns (resp. loops) connected to the white vertices is equal to the number of thorns (resp. loops) connected to the black ones,
%\item the total number of loops connected to the white vertices is equal to the number of loops connected to the black ones,
\item there is a bijection between thorns connected to white vertices and the thorns connected to black vertices. The bijection between thorns will be encoded by assigning the same symbolic {\em latin} labels $\{a,b,c,\ldots\}$ to thorns associated by this bijection,
\item there is a mapping that associates to each loop incident to a white (resp. black) vertex, a black (resp. white) vertex ${\rm v}$  such that the number of white (resp. black) loops associated to a fixed black (resp. white) vertex ${\rm v}$ is equal to its number of incident loops. We will use symbolic {\em greek} labels $\{\alpha,\beta,\ldots\}$ to associate loops with vertices except for the maximal loop (i.e. the loop whose rightmost extremity is the rightmost descendant of the considered vertex) of a root vertex ${\rm r}$ of the non-seed trees. In this case, we draw an arrow (\begin{tikzpicture} \draw[very thick,densely dashed] [->] (0,0) to (0.8,0);   \end{tikzpicture}) outgoing from the root vertex ${\rm r}$ and incoming to the vertex associated with the loop. Arrows are non ordered, and :
\item the ascendant/descendant structure defined by the edges of the forest and the arrows defined above is a tree structure rooted in the root of the seed tree.
\end{itemize}
Finally the degree ${\bf A}$ of the forest is given in the following way:
\begin{itemize}
\item[(vii)] $P_{ij}$ (resp $P'_{ij}$) counts the number of non root white vertices (resp. white root vertices excluding the root of the seed tree) of degree $i$ with a total number of $j$ loops,
\item[(viii)] $Q_{ij}$ (resp $Q'_{ij}$) counts the number of non root black vertices (resp. black root vertices) of degree $i$ with a total number of $j$ loops.
\end{itemize}
\end{definition}

\begin{example}
As an example, Figure \ref{fig:exrecons} depicts two permuted forests. The one on the left is of degree ${\bf A}=(P,P',Q,Q')$ for $ E_{3,1}+E_{2,0}$, $P' = E_{3,1}$, $Q= E_{5,1} + E_{4,1}$, and  $Q' = E_{3,1}$ while the one on the right is of degree ${\bf A^{(2)}}=(P^{(2)},P'^{(2)},Q^{(2)},Q'^{(2)})$ for $P^{(2)}= E_{4,1}$, $P'^{(2)}=\{0\}_{i,j}$, $Q^{(2)}=E_{7,2}$, and $Q'^{(2)}=E_{4,2}$.

\begin{figure}[htbp]
  \begin{center}
    \includegraphics[width=60mm]{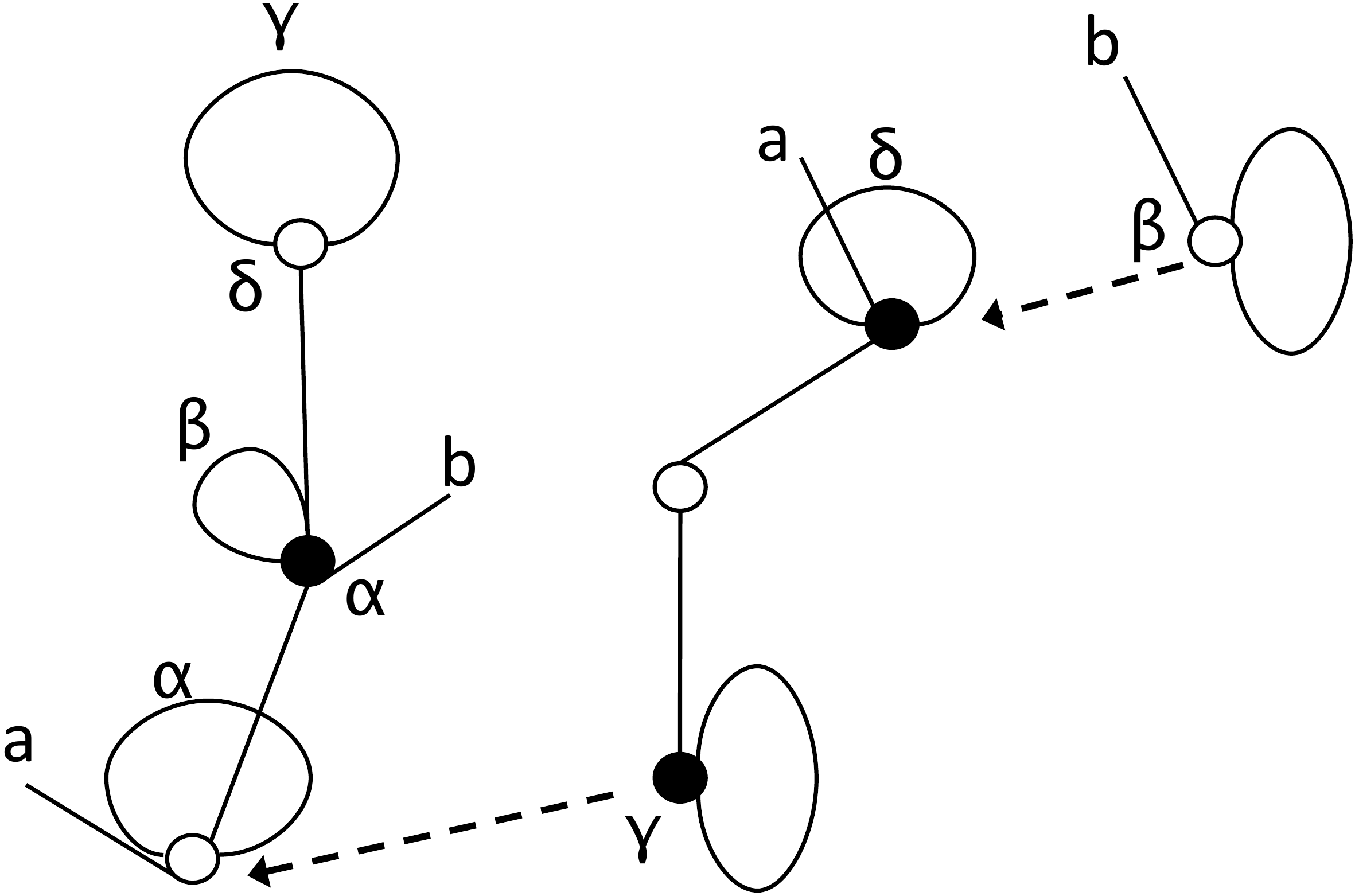}\hspace{15mm}
    \includegraphics[width=30mm]{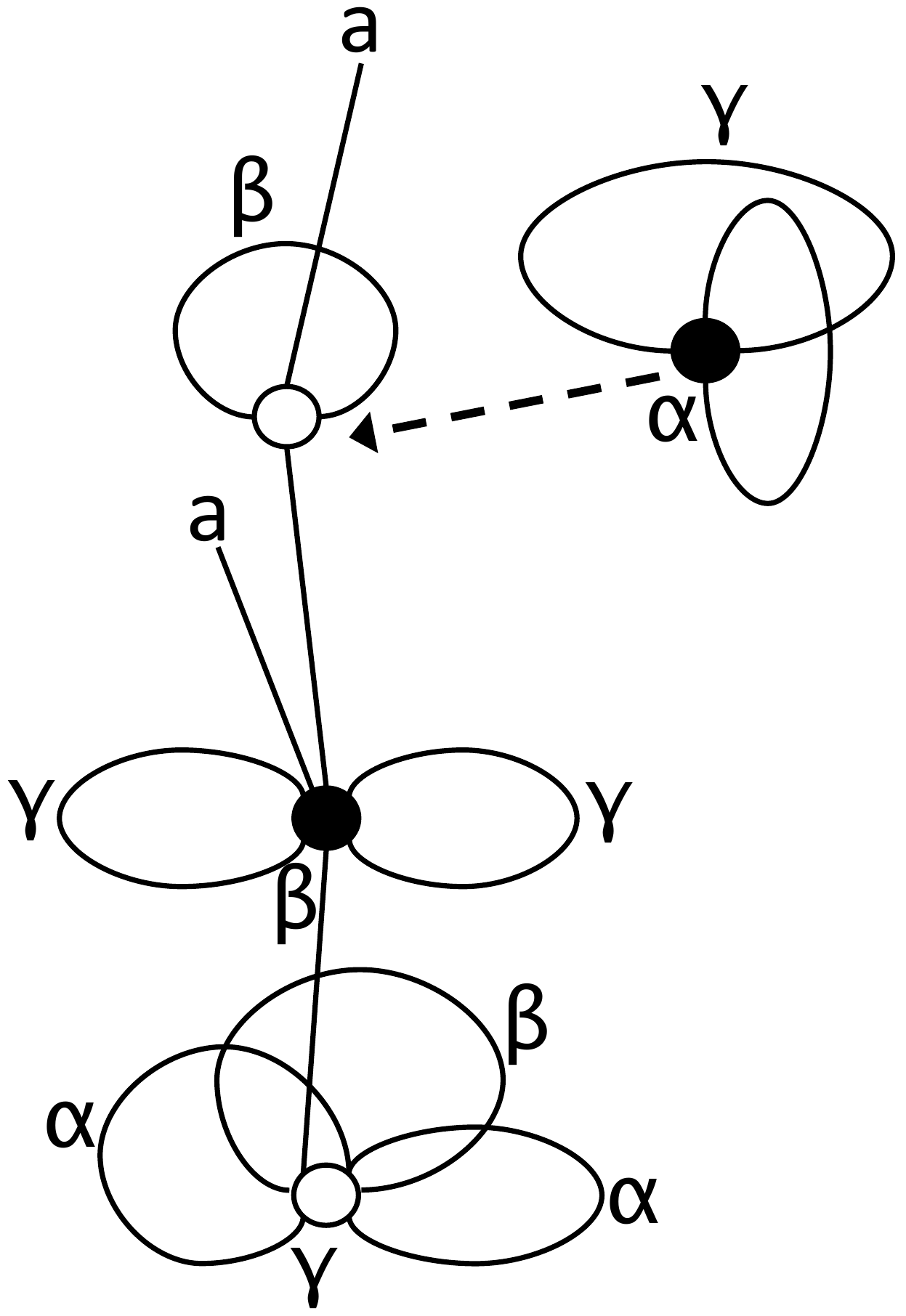}
    \caption{Two Permuted Forests}
    \label{fig:exrecons}
  \end{center}
\end{figure}
\end{example}
\begin{lemma}
\label{lem: lag}
Let $F({\bf A})$ be the cardinality of the set of forests $\mathcal{F}({\bf A})$ defined above. We have:
%Using the Lagrange theorem for implicit functions, one can show:
\begin{eqnarray}
\label{eq:degrees}
F({\bf A}) = \frac{\mathcal{I}({\bf A})}{{\bf A}!}\frac{r!^2(n-q-2r)!(n-1-p-2r)!}{2^{2r-p'-q'}(n-p-q-2r)!}\prod_{i,j}{\binom{i-1}{j,j}}^{(P+Q)_{i,j}}{\binom{i-1}{j,j-1}}^{(P'+Q')_{i,j}}
\end{eqnarray}
Let $F_{p,p',q,q',r}$ be the total number of forests with degree ${\bf A} = (P,P',Q,Q')$ for $p=|P|+1$ (in the following formula we count the root of the seed tree as an internal white vertex), $p'=|P'|$, $q=|Q|$, $q'=|Q'|$ and $r = \sum_j j(Q + Q')_{i,j}$. We have :
\begin{eqnarray}
\label{eq:nodegrees}
F_{p,p',q,q',r} = \frac{n!}{p!p'!q!q'!}\binom{n+2r-1}{p+2r-1,q+2r-1}{\binom{n+2r-1}{r,r}}^{-1}2^{2r-p'-q'}\alpha_{r,p,q,p',q'}
\end{eqnarray}
\normalsize
\end{lemma}
\begin{proof}
The proof is postponed to Annex~\ref{sec: ann}.
\end{proof}
\noindent {\bf Reformulation of the main theorem}\\
In order to show Theorem \ref{thm:main} the next sections are dedicated to the proof of the following stronger result:
\begin{theorem}
\label{thm:ref}
There is a bijection $\Theta_{\bf A}:\mathcal{LP}({\bf A}) \to \mathcal{F}({\bf A}$) and $LP({\bf A}) = F({\bf A})$.
\end{theorem}
\noindent Theorem \ref{thm:comp} is a direct consequence of the above result by setting $r=0$ and using Equation \ref{eq: or}. Using Equation \ref{eq:nodeg}, corollaries \ref{thm:cor} and \ref{thm:corcomp} are also direct consequences of Theorem \ref{thm:ref}. 

\section{Bijection between partitioned locally orientable unicellular hypermaps and permuted forests}
We proceed with the description of the bijective mapping $\Theta_{\bf A}$ between partitioned locally orientable hypermaps and permuted forests of degree ${\bf A}$. Let $(f_3,\pi_1,\pi_2)$ be a partitioned hypermap in $\mathcal{LP}({\bf A})$. The
first step is to define a set of white and black vertices with labeled ordered half edges such that:
\begin{itemize}
\item each white vertex is associated to a block of $\pi_1$ and each black vertex is associated to a block of $\pi_2$,
\item the number of half edges connected to a vertex is half the cardinality of the associated block, and
\item the half edges connected to the white (resp. black) vertices are labeled with the non hat (resp. hat) integers in the associated blocks so that moving clockwise around the vertices the integers are sorted in increasing order.
\end{itemize}
Then we define an ascendant/descendant structure on the vertices. A black vertex $b$ is the descendant of a white one $w$ if the maximum half edge label of $b$ belongs to the block of $\pi_1$ associated to $w$. Similar rules apply to define the ascendant of each white vertex except the one containing the half edge label $1$.\\
If black vertex $b^d$ (resp. white vertex $w^d$) is a descendant of white vertex $w^a$ (resp. black vertex $b^a$) and has maximum half edge label $m$ such that $f_3(m)$ is the label of a half edge of $w^a$ (resp. $b^a$), i.e. $f_3(m^b)$ is a non hat (resp. hat) number, then we connect these two half edges to form an edge. Otherwise $f_3(m)$ is a hat (resp. non hat) number and we draw an arrow (\begin{tikzpicture} \draw[very thick,densely dashed] [->] (0,0) to (0.8,0);   \end{tikzpicture}) between the two vertices. Note that descending edges are ordered but arrows are not.\\ 
\begin{lemma} \label{lem:tree}The above construction defines a tree structure rooted in the white vertex with half edge $1$. \end{lemma}
\begin{proof}
Let black vertices $b_1$ and $b_2$ associated to blocks $\pi_2^{b_1}$ and $\pi_2^{b_2}$ be respectively a descendant and the ascendant of white vertex $w$ associated to $\pi_1^w$. We denote by $m^{b_1}$, $m^{b_2}$ and $m^{w}$ their respective maximum half edge labels (hat, hat, and non hat) and assume $m^{b_1} \neq \widehat{n}$. As $\pi_1^{w}$ is stable by $f_1$, then $f_1(m^{b_1})$ is a non hat number in $\pi_1^w$ not equal to $1$. It follows that $m^{b_1} < f_1(m^{b_1}) \leq m^w < f_2(m^w)$. Then as  $\pi_2^{b_2}$ is stable by $f_2$, it contains $f_2(m^w)$ and $f_2(m^w) \leq m^{b_2}$. Putting everything together yields $m^{b_1} < m^{b_2}$. In a similar fashion, assume white vertices $w_1$ and $w_2$ are descendant and ascendant of black vertex $b$. If we note $m^{w_1}$, $m^{w_2}$ and $m^{b}$ their maximum half edge labels (non hat, non hat, and hat) with $m^{b} \neq \widehat{n}$, one can show that $m^{w_1} < m^{w_2}$. Finally, as $f_1(\widehat{n}) =1$, the black vertex with maximum half edge $\widehat{n}$ is descendant of the white vertex containing the half edge label $1$. \end{proof}
\begin{example}
Using the hypermap of Figure \ref{fig:example} we get the set of vertices and ascendant/descendant structure as described on Figure \ref{fig:verttree}.
\begin{figure}[htbp]
  \begin{center}
    \includegraphics[width=0.8\textwidth]{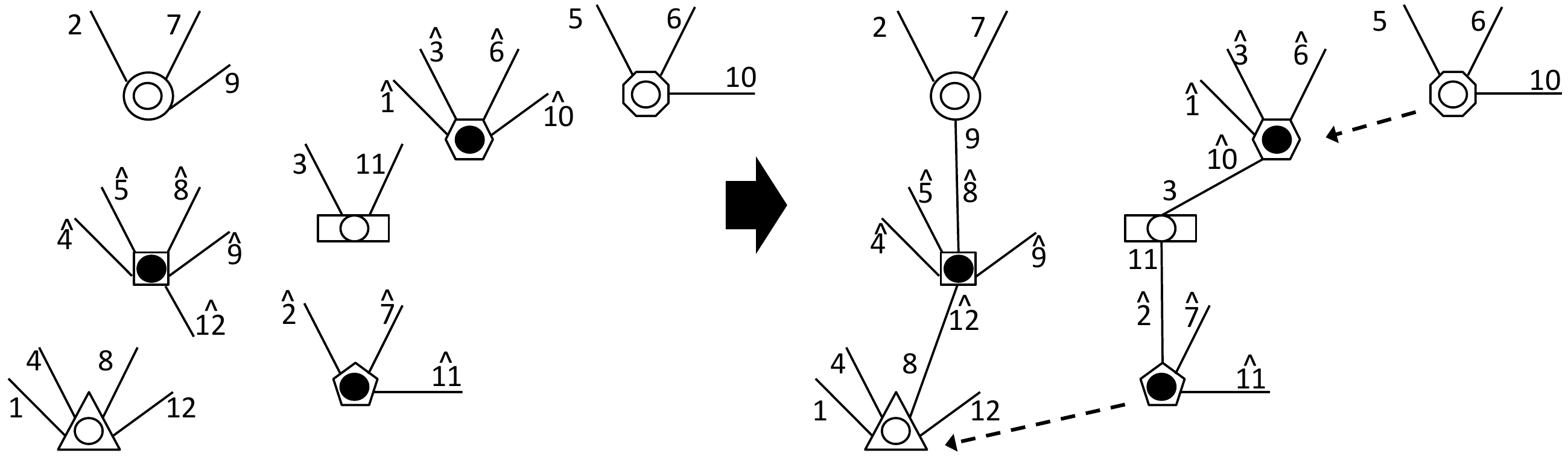}
    \caption{Construction of the ascendant/descendant structure}
    \label{fig:verttree}
  \end{center}
\end{figure}
\end{example}
Next we proceed by linking half edges connected to the same vertex if their labels are paired by $f_3$ to form loops.
Furthermore, we assign greek symbolic labels from $\{\alpha,\beta,\ldots\}$ to all the non maximal loops and the applicable vertices in the following way: 
\begin{itemize}
\item if $i$ and $f_3(i)$ are the numeric labels of a non maximal loop connected to a white (resp. black) vertex, we assign the same label to the loop and the black (resp. white) vertex associated to the block of $\pi_2$ (resp. $\pi_1$) also containing $i$ and $f_3(i)$, 
\item a vertex has at most one such label.
\end{itemize}
As a natural consequence of these two conditions, several loops may share the same label. 
\begin{lemma}\label{lem:loops} The number of loops connected to the vertex labeled $\alpha$ is equal to its number of incoming arrows plus the number of loops  labeled $\alpha$ incident to other vertices in the forest.
 \end{lemma}
\begin{proof} The result is a direct consequence of the fact that in each block the number of hat/hat pairs is equal to the number of non hat/non hat pairs.\end{proof} 

As a final step we define a bijection between the remaining half edges (thorns) connected to the white vertices and the ones connected to the black vertices. If two remaining thorns are paired by $f_3$ then these two thorns are given the same label from $\{a,b,\ldots\}$.  Then all the original integer labels are removed.
Denote by $\widetilde{F}$ the resulting forest.
\begin{example}
We continue with the hypermap from Figure \ref{fig:example} and perform the final steps of the construction as described on Figure \ref{fig:loopsperm} (note that the geometric shapes are here for reference only, they do not play any role in the final object $\widetilde{F}$).
\begin{figure}[htbp]
  \begin{center}
    \includegraphics[width=0.8\textwidth]{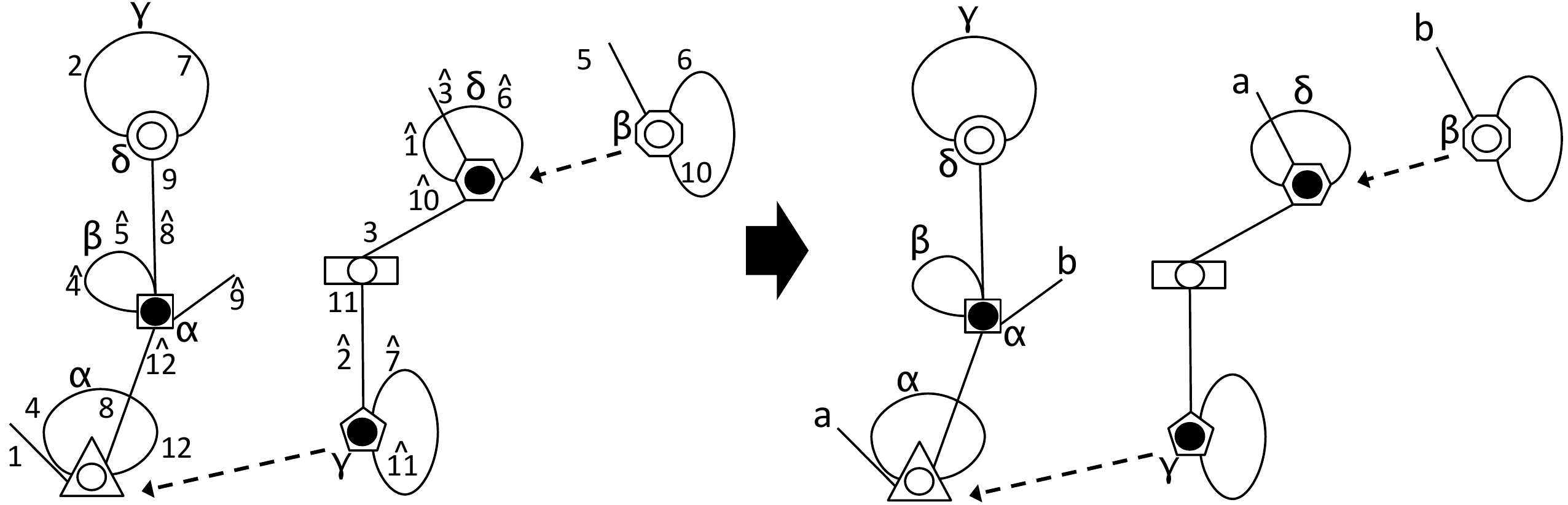}
    \caption{Final steps of the permuted forest construction}
    \label{fig:loopsperm}
  \end{center}
\end{figure}
\end{example}

\noindent As a direct consequence of definition \ref{def:forest}, $\widetilde{F}$ belongs to $\mathcal{F}(A)$.

\section{Proof of the bijection}
We show that mapping $\Theta_{\bf A}:(f_3,\pi_1,\pi_2) \mapsto \widetilde{F}$ is indeed one-to-one.
\subsection{Injectivity}
We start with a forest $\widetilde{F}$ in $\mathcal{F}({\bf A})$ and show that there is at most one triple $(f_3,\pi_1,\pi_2)$ in $\mathcal{LP}({\bf A})$ such that $\Theta_{\bf A}(f_3,\pi_1,\pi_2) = \widetilde{F}$. The first part is to notice that within the construction in $\Theta_{\bf A}$ the original integer label of the leftmost descendant (thorn, half loop or edge) of the root vertex of the seed tree is necessarily $1$ (this root is the vertex containing $1$ and the labels are sorted in increasing order from left to right).\\
Assume we have recovered the positions of integer labels $1,\widehat{1},2,\widehat{2},\ldots,i$, for some $1 \leq i \leq n-1$, non hat number. Then four cases can occur:
\begin{itemize}
\item $i$ is the integer label of a thorn of latin label $a$. In this case $f_3(i)$ is necessarily the integer label of the thorn connected to a black vertex also labeled with $a$. But as the blocks of $\pi_2$ are stable by both $f_3$ and $f_2$ then $\widehat{i} = f_2(i)$ is the integer label of one of the descendants of the black vertex with thorn $a$. As these labels are sorted in increasing order, necessarily, $\widehat{i}$ labels the leftmost descendant with no recovered integer label,
\item $i$ is the integer label of a half loop of greek label $\alpha$. Then, in a similar fashion as above $\widehat{i}$ is necessarily the leftmost unrecovered integer label of the black vertex with symbolic label $\alpha$,
\item $i$ is the integer label of a half loop with no symbolic label (i.e, either $i$ or $f_3(i)$ is the maximum label of the considered white vertex). Then, $\widehat{i}$ is necessarily the leftmost unrecovered integer label of the black vertex at the other extremity of the arrow outgoing from the white vertex containing integer label $i$,
\item $i$ is the integer label of an edge and $\widehat{i}$ is necessarily the leftmost unrecovered integer label of the black vertex at the other extremity of this edge.
\end{itemize}
Finally, using similar four cases for the black vertex containing the descendant with integer label $\widehat{i}$ and the fact that blocks of $\pi_1$ are stable by $f_3$ and $f_1$, the thorn, half loop or edge with integer label $i+1 = f_1(\widehat{i})$ is uniquely determined as well.

We continue with the procedure described above until we fully recover all the original labels $[n]\cup [\hat{n}]$. According to the construction of $\widetilde{F}$ the knowledge of all the integer labels uniquely determines the blocks of $\pi_1$ and $\pi_2$. The pairing $f_3$ is uniquely determined by the loops, edges and thorns with same latin labels as well.

\begin{example}
Assume the permuted forest $\widetilde{F}$ is the one on the right hand side of Figure \ref{fig:exrecons}. The steps of the reconstruction are summarized in Figure \ref{fig:recons}. We get that the unique triple $(f_3,\pi_1,\pi_2)$ such that $\Theta_{\bf A}(f_3,\pi_1,\pi_2) = \widetilde{F}$ is:
\begin{eqnarray}
\nonumber f_3 &=& (1\,\, 4)(\widehat{1}\,\, \widehat{8})(2\,\, 9)(\widehat{2}\,\, \widehat{3})(3\,\, \widehat{11})(\widehat{4}\,\, \widehat{10})(5\,\, 7)(\widehat{5}\,\, 6)(\widehat{6}\,\, 11)(\widehat{7}\,\, \widehat{9})(8\,\, 10)\\
\nonumber \pi_1 &=& \{\{\widehat{11},1,\widehat{1},2,\widehat{2},3,\widehat{3},4,\widehat{7},8,\widehat{8},9,\widehat{9},10\};\{\widehat{4},5,\widehat{5},6,\widehat{6},7,\widehat{10},11\}\}\\
\nonumber \pi_2 &=& \{\{2,\widehat{2},3,\widehat{3},5,\widehat{5},6,\widehat{6},7,\widehat{7},9,\widehat{9},11,\widehat{11}\};\{1,\widehat{1},4,\widehat{4},8,\widehat{8},10,\widehat{10}\}\}
\end{eqnarray}
\begin{figure}[htbp]
  \begin{center}
    \includegraphics[width=1\textwidth]{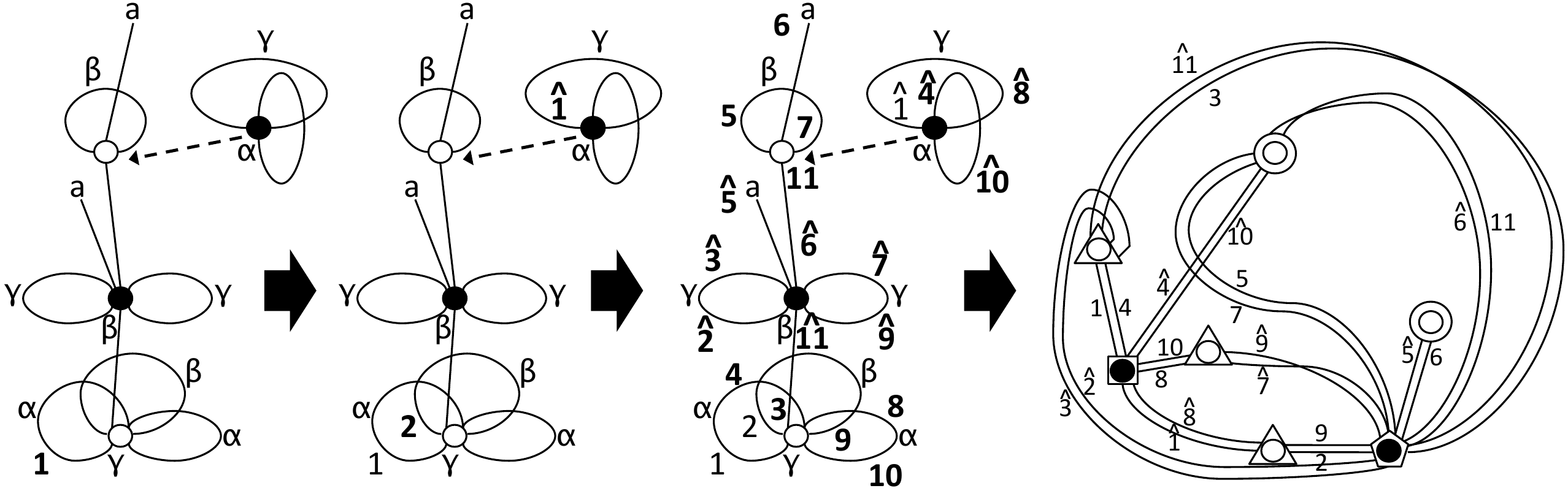}
    \caption{Recovery of the integer labels and the partitioned map}
    \label{fig:recons}
  \end{center}
\end{figure}
\end{example}

\subsection{Surjectivity}
To prove that $\Theta_{\bf A}$ is surjective, we have to show that the reconstruction procedure of the previous section always finishes with a valid output.\\ 
Assume the procedure comes to an end at step $i$ before all the integer labels are recovered (where $i$ is for example non hat, the hat case having a similar proof). It means that prior to this step we have already recovered all the labels of vertex $v^i$ identified as the one containing $\widehat{i}$ (or $i+1$). This is impossible by construction provided $v^i$ is not the root vertex of the seed tree. Indeed the number of times a vertex is identified for the next step is equal to its number of thorns plus its number of edges plus twice the number of loops that have the same greek label as $v^i$ plus twice the incoming arrows. Using Property {\em (iv)} of Definition \ref{def:forest}, we have that the sum of the two latter numbers is twice the number of loops of $v^i$. As a consequence, the total number of times the recovering process goes through $v^i$ is exactly (and thus never more than) the degree of $v^i$.

If $v$ is the root vertex of the seed tree the situation is slightly different due to the fact that we recover label $1$ before we start the procedure. To ensure that the procedure does not terminate prior to its end, we need to show that the $\mid v \mid$-th time the procedure goes through the root vertex is right after all the labels of the forest have been recovered. Again, this is always true because:
\begin{itemize}
\item the last element of a vertex to be recovered is the label of the maximum element of the associated block. Consequently, all the elements of a vertex are recovered only when all the elements of the descending vertices (through both arrows and edges) are recovered.
\item property {\em (v)} of Definition \ref{def:forest} states that the ascendant/descendant structure involving both edges and arrows is a tree rooted in $v$. As a result, the procedure goes the $v$-th time through $v$ only when all the elements of all the other vertices are recovered.
\end{itemize} 

%%%%%%%%%%%%%%%%%%%%%%%%%%%%%%%%%%%%%%%%%%%%%%%%%%%%%%%%%%%%%

\section{Additional results} 
The bijection proved in the previous sections may be used directly to derive efficiently some additional results that may not be obvious from the formula in Theorem \ref{thm:main}.
\subsection{Coefficient of $m_\la(X)m_n(Y)$}

Using the bijection between partitioned locally orientable hypermaps and permuted forests, one can show:
\begin{theorem} For $\la \vdash n$ the coefficient of $m_\la(X)m_n(Y)$ in the monomial expansion of $P^{\mathbb{R}}_n(X,Y)$ is given by
\begin{equation}
\label{eq: mlamn}
[m_\la(X)m_n(Y)]P^{\mathbb{R}}_n(X,Y) =  \binom{n}{\la}(2\la-1)!!,
\end{equation}
where $(2\la-1)!! = \prod_i(2\la_i-1)!!$ and $(2\la_i-1)!! = (2\la_i-1)(2\la_i-3)\dots1$.
\end{theorem}
\begin{proof}
Let $F_n$ be the number of forests composed of exactly one white (the root of the seed tree) and one black vertex. Obviously $F_n = (2n-1)!!$ as any such forest is fully described as a pairing of the $2n$ children around the white and the black vertex (a loop is the pairing of two children of the same vertex, thorns with latin letters are pairings of one black and one white child and an edge is the pairing of the rightmost child of the black vertex to one of the child of the white root.)\\ 
\indent But a forest with one white (root) vertex and $\ell(\lambda)$ black vertices of degree distribution $\la$ ($F_\la$ denotes the number of such forests) can be seen as a $\ell(\lambda)$-tuple of forests with one white and one black vertex of degree $\{\la_i\}_{1\leq i \leq \ell(\la)}$. The $i$-th forest is composed of the $i$-th black vertex with its descendants and one new white vertex with a subset of descendants of the original one's containing:
\begin{itemize} 
\item(i) the edge linking the white vertex and the $i$-th black vertex (if any), 
\item(ii) the thorns in bijection with the thorns of the $i$-th black vertex,
\item(iii) the loops mapped to the $i$-th black vertex.
\end{itemize} 
The construction is bijective if we distinguish in the initial forest the black vertices with the same degree ($Aut_\la$ ways to do it) and we keep track in the tuple of forests the initial positions of the descendants of the white vertices within the initial forest ($\binom{n}{\la}$ possible choices). We get:
\begin{equation}
Aut_\la F_\la = \binom{n}{\la}\prod_i F_{\la_i} =\binom{n}{\la}(2\la-1)!!
\end{equation}
\begin{figure}[h]
\begin{center}
\includegraphics[width=12cm]{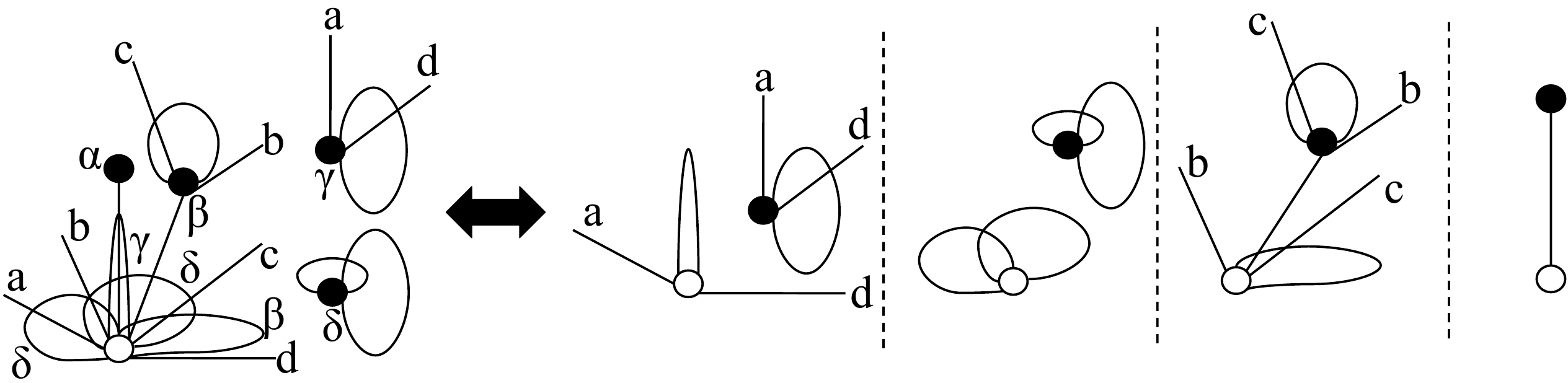}
\end{center}
\caption{Splitting a forest of black degree distribution $\la$ into a $\ell(\la)$-tuple of two vertex forests for $\la = [1^14^25^1]$.}
\label{forests}
\end{figure}
\end{proof}
%Finally, according to \cite{MV11}, $Aut_\la F_\la$ is equal to the desired coefficient in (\ref{eq: mlamn}). 
\begin{rem}
Using the formula of Theorem \ref{thm:main}, we have:
\begin{equation}
\sum_{Q,Q'}\prod_{i,j}\frac{2^{Q'_{ij}-2j(Q_{ij}+Q'_{ij})}}{Q_{ij}!Q'_{ij}!}{\binom{i-1}{j,j}}^{Q_{ij}}{\binom{i-1}{j,j-1}}^{Q'_{ij}} = \frac{(2\la-1)!!}{\la!Aut_\la},
\end{equation}
where the sum runs over two dimensional arrays $Q$ and $Q'$ with $n_i(\lambda)=\sum_{j \geq 0}Q_{ij} + Q'_{ij}$.
\end{rem}
%\begin{rem}
%If we admit the expansion of $Z_n$ in terms of monomials, we can directly show (\ref{eq: mlamn}) as: 
% %the coefficient in $m_n(y)$ in the generating series is:
%\begin{align}
%  {|B_n|}^{-1}\sum_{\la \vdash n}\left ( \sum_{\mu \vdash n}{b_{\la,\mu}^n}\right) p_{\la}(x)=  {|B_n|}^{-1}\sum_{\la \vdash n}|K_\la|p_{\la}(x) =Z_n(x)= \sum_{\la \vdash n}\binom{n}{\la}(2\la-1)!!m_\la(x)
%\end{align}
%
%\end{rem}
\subsection{Coefficient of $m_{n-a,1^{a}}(X)m_{n-a,1^{a}}(Y)$}
The number $F_{(n-a,1^a),(n-a,1^a)}$ of forests with $a+1$ white (including the root of the seed tree) and $a+1$ black vertices, both of degree distribution $(n-a,1^{a})$, can be easily obtained from the number of two-vertex forests $F_{n-2a}$. We consider $2a\leq n-1$, it is easy to show that the coefficient is equal to $0$ otherwise. Two cases occur: either the white vertex with degree $n-a$ is the root and there are $\binom{n-a}{a}\times\binom{n-a-1}{a}$ ways to add the black and the white descendants of degree $1$, or the root is a white vertex of degree $1$ and there are $\binom{n-a-1}{a-1}\times\binom{n-a-1}{a}$ ways to add the remaining white vertices and the $a$ black vertices of degree $1$ (see Figure \ref{forests2}). We have:
\begin{equation}
F_{(n-a,1^a),(n-a,1^a)} = F_{n-2a}\binom{n-a-1}{a}\left [\binom{n-a}{a}+\binom{n-a-1}{a-1}\right ]
\end{equation}
\begin{figure}[h]
\begin{center}
\includegraphics[width=8cm]{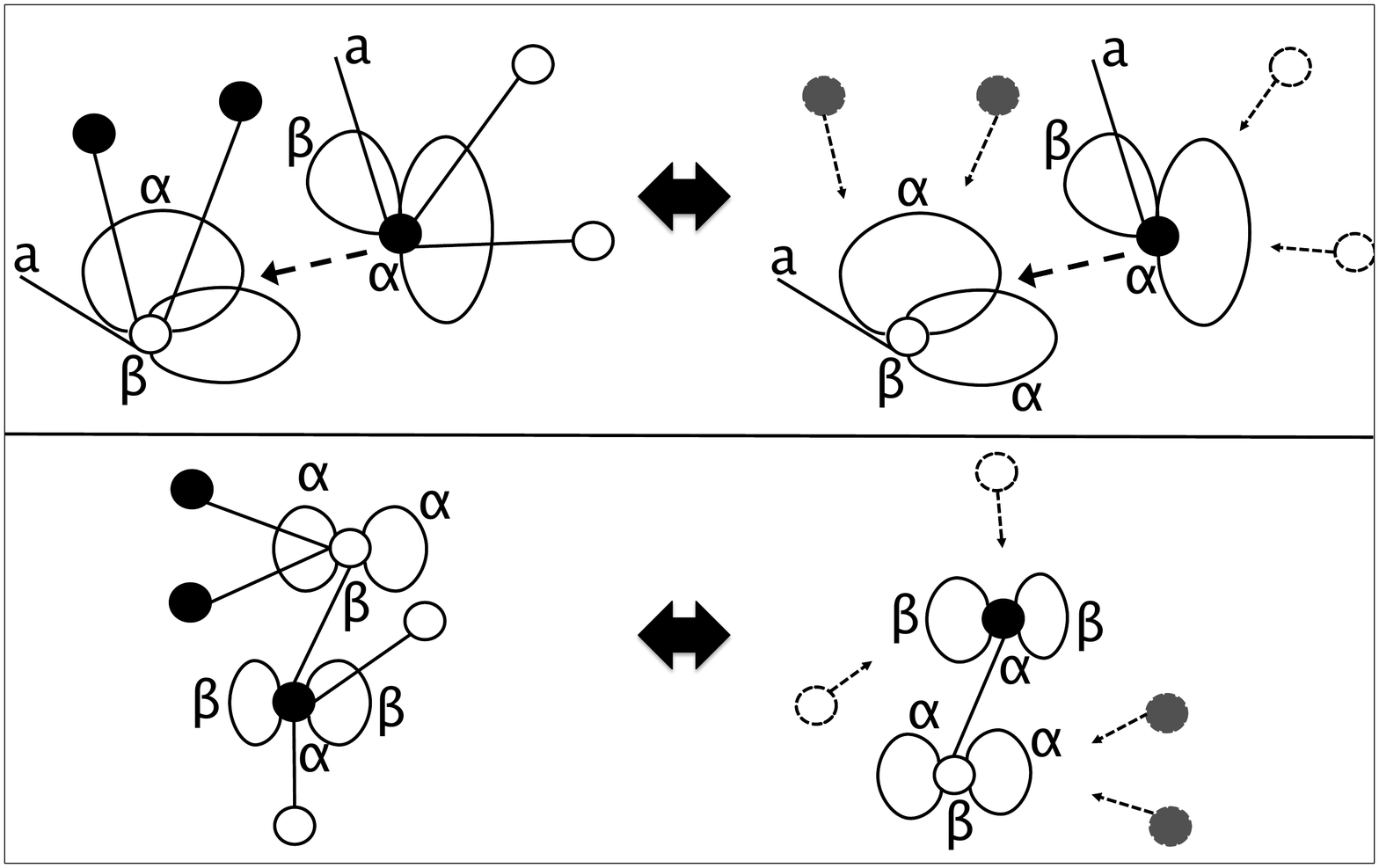}
\end{center}
\caption{Two possible decompositions of forests of degree $(n-a,1^{a})$.}
\label{forests2}
\end{figure}
As a result, we obtain :
\begin{theorem} For {\em $a$} non negative integer such that $2a\leq n-1$, the coefficient of $m_{n-a,1^{a}}(X)m_{n-a,1^{a}}(Y)$ in the monomial expansion of $P^{\mathbb{R}}_n(X,Y)$ is given by:
\begin{align}
\nonumber [m_{n-a,1^{a}}(X)m_{n-a,1^{a}}(Y)]P^{\mathbb{R}}_n(X,Y) &= Aut_{n-a,1^{a}}^2F_{(n-a,1^a),(n-a,1^a)}\\
&= n(n-2a)\left(\frac{(n-a-1)!}{(n-2a)!}\right)^2(2n-4a-1)!!
\end{align}
\end{theorem}
%%%%%%%%%%%%%%%%%%%%%%%%%%%%%%%%%%%%%%%%
\section{Annex: enumeration of permuted forests} 
\label{sec: ann}
\subsection{General considerations}
\label{subsec: andeg}
In this section we prove Lemma \ref{lem: lag} and compute the cardinality of the set $\mathcal{F}({\bf A})$. To this extent we slightly modify the considered forests and define the set $\mathcal{G}({\bf A})$ of cardinality $G({\bf A})$. The definition of these forests differs from the one of $\mathcal{F}({\bf A})$ as in $\mathcal{G}({\bf A})$:
\begin{itemize}
\item[(i)] There is no bijection between the thorns connected to the black vertices and the one connected to the white vertices.
\item[(ii)]  All the non seed trees with a white (resp. black) root are labeled by an integer in $\{1,\ldots,p'\}$ (resp. $\{1,\ldots,q'\}$).
\item[(iii)]  All the loops incident to a white (resp. black) vertex whose right extremity is not the rightmost descendant of the root vertex of a non seed-tree are labeled with an integer of $\{1,\ldots, r-p'\}$ (resp. $\{1,\ldots, r-q'\}$) according to the labeling of the non seed trees. If $k_i$ $(0\leq i \leq p')$ (resp. $(1\leq i \leq q')$) is the number of such loops in white (resp . black) rooted tree $i$ (tree $0$ is the seed tree) we use integers $\sum_{j\leq i-1}k_j+1, \sum_{j\leq i-1}k_j +2,\ldots, \sum_{j\leq i-1}k_j+k_i$ to label these loops. Within a tree loops are labeled in a classical order, say according to the depth first traversal of the tree.  
\item[(iv)] Instead of the "coloration" of the loops with the greek letters, additional unordered labeled arrows are connected to the black and white vertices. The labels of the arrows connected to a given vertex $v$ are the ones of the loops colored by $v$.  
\end{itemize}
\begin{example}
The right hand side forest of Figure \ref{fig:lag} belongs to $\mathcal{G}({\bf A})$. The left hand side one is a forest of $\mathcal{F}({\bf A})$ with an equivalent coloration of the loops with the greek letters.
\begin{figure}[htbp]
  \begin{center}
    \includegraphics[width=0.7\textwidth]{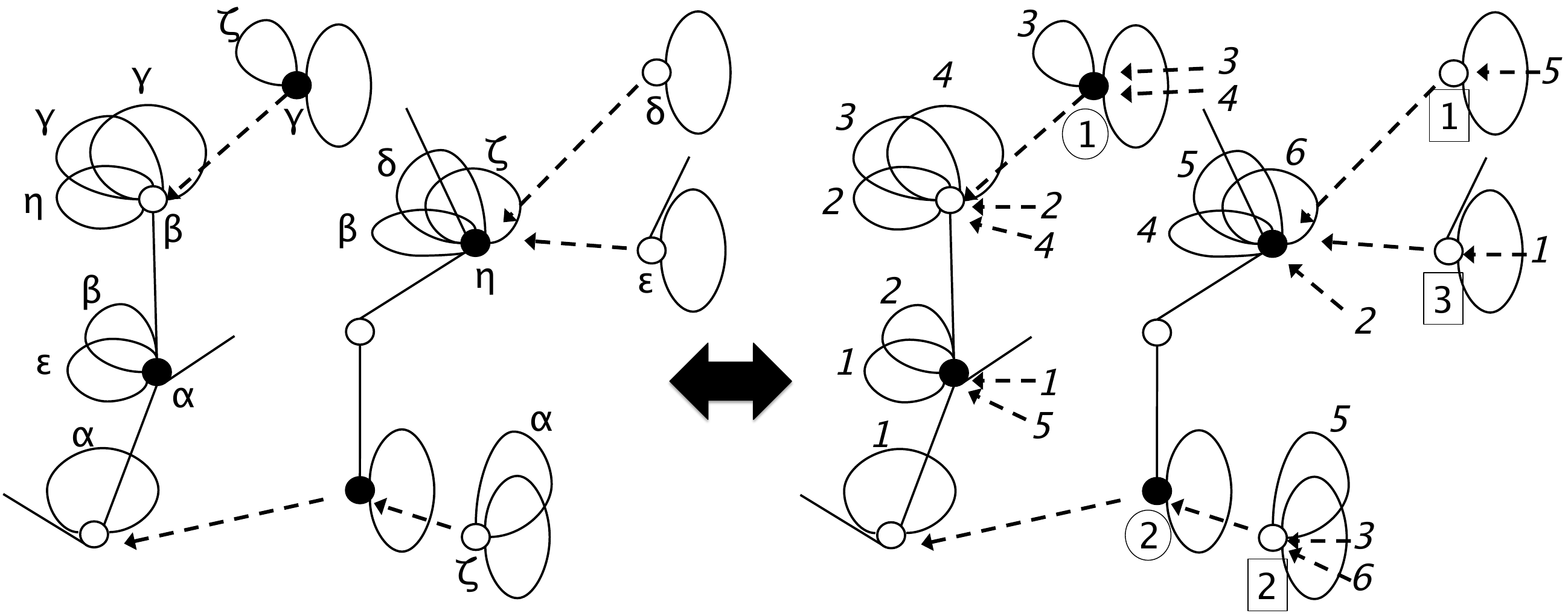}
    \caption{A forest of $\mathcal{G}({\bf A})$ (right) and its "equivalent" in $\mathcal{F}({\bf A})$ (left).}
    \label{fig:lag}
  \end{center}
\end{figure}
\end{example}
\begin{lemma}
\label{lem: f2g}
$\mathcal{G}({\bf A})$ as defined above is in a $p'!q'!$ to $(n-p-q-2r)!$ relation with $\mathcal{F}({\bf A})$: $$F({\bf A}) = \frac{(n-p-q-2r)!}{p'!q'!}G({\bf A})$$
\end{lemma}
\begin{proof}(Sketch) The factor $(n+1-p-q-2r)!$ clearly comes from the fact that we removed the bijection between the thorns. There are $p'!q'!$ ways of labeling the non-seed trees of a forest in $\mathcal{F}$, i.e. there are at most $p'!q'!$ forests in $\mathcal{G}$ for $(n+1-p-q-2r)!$ forests in $\mathcal{F}$. Then two different labelings of the non seed trees yields two different forests of $\mathcal{G}$. Obviously, permuting the labels of the root tree of some subforests that are non identical up to the various labels and non linked by a dotted arrow to the same vertex automatically reaches distinct forests in $\mathcal{G}$. Next, one can notice that in any subforest there is exactly one more labeled arrow than labeled loops. As a result, a set of  "identical" subforests has labeled arrows indexed by a loop that does not belong to the considered set and permuting their root's labels yields another object of $\mathcal{G}$.
\end{proof}

\subsection{Demonstration of Equation \ref{eq:degrees}}
The computation of $G(A)$ is performed thanks to the multivariate Lagrange theorem for implicit functions. For given partitions $\lambda, \mu \vdash n$ and integer $r \geq 0$ we consider the generating function $H$:
\begin{equation}
H = x_0\sum_{{\bf A}\in M^r_{\lambda,\mu}} G({\bf A}) x^{p}y^{q}\frac{x'^{p'}}{p'!}\frac{y'^{q'}}{q'!}\frac{{f_1}^{r-q'}}{(r-q')!}\frac{{f_2}^{r-p'}}{(r-p')!}\mathbf{t}^{P}\mathbf{t'}^{P'}\mathbf{u}^{Q}\mathbf{u'}^{Q'}.
\end{equation}
Variables $x_0$, $x$, $x'$, $y$, $y'$, $f_1$ and $f_2$ mark respectively the root of the seed tree, non root white vertices, root white vertices (excluding the root of the seed tree), non root black vertices, root black vertices, labeled arrows incident to white vertices and labeled arrows incident to black vertices.  Furthermore, $\mathbf{t}$, $\mathbf{t'}$, $\mathbf{u}$, $\mathbf{u'}$ are two dimensional indeterminate such that
$$\mathbf{a}^{X} = \prod_{i,j} a_{ij}^{X_{ij}}$$
\noindent for $a \in \{t,t',u,u'\}$ and $X \in \{P,P',Q,Q'\}$.

We define $W$, $W'$, $B$ and $B'$ as the generating functions of subforests descending from an internal white, root white (excluding the root of the seed tree), internal black and root black vertex.
\begin{remark}
According to the above definition the degree of the root vertex of a subforest marked by $W$ is one plus the number of descendants while the degree of the root vertex marked by $H$ is only the number of its descendants. As a result, $W \neq H$.
\end{remark}
\noindent Additionally, we note $A_w$ and $A_b$ the generating functions of the labeled arrows incident to the white (respectively black) vertices. Trivially, $A_w = f_1$ and $A_b = f_2$. According to the construction rules of the considered permuted forest, we have the following relation between $W$ and the other considered generating functions:
\begin{equation}
W = x\sum_{i\geq 1,j,k\geq 0}t_{i,j}\binom{i-1}{2j}(1+B)^{i-1-2j}(2j-1)!!\frac{B'^k}{k!}\frac{A_w^{j-k}}{(j-k)!}
\end{equation}
Indeed, assume the degree of an internal white vertex is $i$, its number of incident loops is $j$ and this white vertex has exactly $k$ descending non-seed trees. This vertex has $i-1$ ordered children. Among them $i-1-2j$ can be either a thorn or an edge also incident to an internal black vertex. The remaining $2j$ are the extremities of loops that can be paired in $(2j-1)!!$ different ways. Then $j$ loops and $k$ incident non seed trees necessarily implies $j-k$ incident labeled arrows. The factors $1/k!$ and $1/(j-k)!$ are needed as descending non seed trees and incoming labeled arrows are not ordered.  This formula simplifies using $(2j-1)! = 2^{-j}(2j)!/j!$ and $\sum_{0\leq k \leq j}B'^kA_w^{j-k}/k!(j-k)! = (B'+A_w)^j/j!$. One gets:
\begin{equation*}
W = x\sum_{i\geq 1,j \geq 0}t_{i,j}\binom{i-1}{j,j}\left (\frac{B'+A_w}{2}\right)^j(1+B)^{i-1-2j}.
\end{equation*}
We define function $\Phi_W$ as:
\begin{equation*}
W = x\,\Phi_W(H,W,B,W',B',A_w,A_b).
\end{equation*}
Similarly,
\begin{align*}
B &= y\sum_{i\geq 1,j \geq 0}u_{i,j}\binom{i-1}{j,j}\left (\frac{W'+A_b}{2}\right)^j(1+W)^{i-1-2j},\\
B &= y\,\Phi_B(H,W,B,W',B',A_w,A_b). 
\end{align*}
Also,
\begin{equation*}
W'= x'\sum_{i\geq 1,j,k\geq 0}t'_{i,j}\binom{i-1}{2j-1}(1+B)^{i-2j}(2j-1)!!\frac{B'^k}{k!}\frac{A_w^{j-k}}{(j-k)!}.
\end{equation*}
Then
\begin{align*}
W' &= 2x'\sum_{i\geq 1,j \geq 0}t'_{i,j}\binom{i-1}{j,j-1}\left (\frac{B'+A_w}{2}\right)^j(1+B)^{i-2j},\\
W &= x'\,\Phi_{W'}(H,W,B,W',B',A_w,A_b).
\end{align*}
Similar equation applies to $B'$:
\begin{align*}
B' &= 2y'\sum_{i\geq 1,j \geq 0}u'_{i,j}\binom{i-1}{j,j-1}\left (\frac{W'+A_b}{2}\right)^j(1+W)^{i-2j},\\
B' &= y'\,\Phi_{B'}(H,W,B,W',B',A_w,A_b).
\end{align*}
Generating function $H$ verifies the following equation:
\begin{align*}
H &= x_0\binom{i_0}{j_0,j_0}\left (\frac{B'+A_w}{2}\right)^{j_0}(1+B)^{{i_0}-2j_0},\\
H &= x_0\,\Phi_{H}(H,W,W',B,B',A_w,A_b).
\end{align*}
Define $\Phi_{A_w} = \Phi_{A_b} = 1$ and ${\bf \Phi} = (\Phi_H,\Phi_W,\Phi_{W'},\Phi_B,\Phi_{B'},\Phi_{A_w},\Phi_{A_b}) =(\Phi_1,\Phi_2,\ldots,\Phi_7)$ and ${\bf x} = (x_0,x,x',y,y',f_1,f_2)$. We have
\begin{equation}
(H,W,B,W',B',A_w,A_b) = {\bf x}\,{\bf \Phi}(H,W,B,W',B',A_w,A_b)
\end{equation}
\noindent Using the multivariate Lagrange inversion formula for monomials (see \cite[1.2.9]{GJCE}) we find:
\begin{eqnarray}
\nonumber {k_1 k_2 k_3 k_4 k_5 k_6 k_7}\,[{\bf x}^{\bf k}]\,\,H =\hspace{0mm}
\label{eq lagrange}\sum_{\{\mu_{ij}\}}\mid\mid \delta_{ij}k_j-\mu_{ij}\mid\mid\prod_{1\leq i \leq 7}[H^{\mu_{i1}}W^{\mu_{i2}}W'^{\mu_{i3}}B^{\mu_{i4}}B'^{\mu_{i5}}A_w^{\mu_{i6}}A_b^{\mu_{i7}}]{ \Phi}_i^{k_i},
\end{eqnarray}
\noindent where $\mid\mid \cdot \mid\mid$ denotes the determinant, ${\bf k} = (k_1,k_2,k_3,k_4,k_5,k_6,k_7)$, $\delta_{ij}$ is the Kronecker delta function and the sum is over all $7\times 7$ integer matrices $\{\mu_{ij}\}$ such that:

\begin{minipage}{1\linewidth}\centering
\begin{itemize}
\item $\mu_{ij} = 0$ for $i\geq 6$ or $j=1$ or $i=j$
\item $\mu_{12} = \mu_{13} = \mu_{17} =  \mu_{23} = \mu_{27}=\mu_{32}=\mu_{37}=\mu_{45} =\mu_{46} = \mu_{54} =\mu_{56}= 0$
\item $\mu_{42} + \mu_{52} = k_2 $
\item $\mu_{43} + \mu_{53} = k_3$
\item $\mu_{14} + \mu_{24}+\mu_{34} = k_4$
\item $\mu_{15} + \mu_{25}+\mu_{35} = k_5$
\item $\mu_{16} + \mu_{26}+\mu_{36} = k_6$
\item $\mu_{47} + \mu_{57} = k_7$
\end{itemize}
\end{minipage}

We look at the solution when ${\bf k} = (1,p,p',q,q',r-q',r-p')$. Namely when $\mu$ is of the form\\

\begin{minipage}{1\linewidth}
\centering
$\mu=
\left[
\begin{array}{ccccccc}
0 & 0 & 0 & a & b & g&0\\
0 & 0 & 0 & c & d & h&0\\
0 & 0 & 0 & q-a-c & q'-b-d & r-q'-g-h&0\\
0 & e & f & 0 & 0 & 0&i\\
0 & p-e & p'-f & 0 & 0 & 0&r-p'-i\\
0 & 0 & 0 & 0 & 0 & 0&0\\
0 & 0 & 0 & 0 & 0 & 0&0
\end{array}
\right],
$
\end{minipage}
\noindent where the parameters $a,b,c,d,e,f,g,h,i$ are non negative integers.
In this case the determinant $\Delta = \mid\mid \delta_{ij}k_j-\mu_{ij}\mid\mid$ reads:\\

\begin{minipage}{1\linewidth}
\centering
$\Delta=
\left |
\begin{array}{ccccccc}
1 & 0 & 0 & -a & -b & -g&0\\
0 & p & 0 & -c & -d & -h&0\\
0 & 0 & p' & a+c-q & b+d-q' & g+h+q'-r&0\\
0 & -e & -f & q & 0 & 0&-i\\
0 & e-p & f-p' & 0 & q' & 0&i+p'-r\\
0 & 0 & 0 & 0 & 0 & r-q'&0\\
0 & 0 & 0 & 0 & 0 & 0&r-p'
\end{array}
\right |,
$
\end{minipage}
\noindent that immediately reduces as\\

\begin{minipage}{1\linewidth}
\centering
$\Delta=(r-p')(r-q')
\left |
\begin{array}{cccc}
 p & 0 & -c & -d\\
 0 & p' & a+c-q & b+d-q'\\
 -e & -f & q & 0\\
 e-p & f-p' & 0 & q'
\end{array}
\right |. 
$
\end{minipage}
Define $\Delta'=\Delta/(r-p')(r-q')$. Looking at the dependence in $\bf t,t',u,u'$ we have\\
\begin{align}
\label{eq: G}\nonumber G({\bf A}) =
&\frac{{p'!q'!(r-p')!(r-q')!}}{pp'qq'}\sum\Delta' [B^{q-a-c}B'^{q'-b-d}A_w^{r-q'-g-h}{\bf t'}^{P'}]\Phi_{W'}^{p'}\\
&\times[B^{c}B'^{d}A_w^{h}{\bf t}^P]\Phi_W^p[W^{e}W'^{f}A_b^{i}{\bf u}^Q]\Phi_B^q[W^{p-e}{W'}^{p'-f}A_b^{r-p'-i}{\bf u'}^{Q'}]\Phi_{B'}^{q'}[B^{a}B'^{b}A_w^{g}]\Phi_H
\end{align}

\noindent The sum runs over the parameters $a,b,c,d,e,f,g,h,i$. Next step is to compute
\begin{equation*}
\Phi_W^p = \sum_{{\bf \eta} \in co(p)}\binom{p}{{\bf \eta} }\prod_{i,j}\left [t_{i,j}\binom{i-1}{j,j}\left (\frac{B'+A_w}{2}\right)^j(1+B)^{i-1-2j}\right ]^{\eta_{i,j}},
\end{equation*}
where $co(p)$ denotes the sets of the two dimensional compositions $\bf \eta$ of $p$  i.e the two dimensional arrays such that $\sum_{i,j} \eta_{i,j} = p$. We have
\begin{equation*}
\Phi_W^p = \sum_{{\bf \eta} \in co(p)}\binom{p}{{\bf \eta} }\left (\frac{B'+A_w}{2}\right)^{\sum_{i,j}j\eta_{i,j}}(1+B)^{\sum_{i,j}(i-1-2j)\eta_{i,j}}\prod_{i,j}\left [ \binom{i-1}{j,j}t_{i,j}\right ]^{\eta_{i,j}}.
\end{equation*}
Then:
\begin{equation*}
[B^{c}B'^{d}A_w^{h}{\bf t}^P]\Phi_W^p = \frac{p!}{P!}\left (\frac{1}{2}\right)^{\sum_{i,j}jP_{i,j}}\binom{\sum_{i,j}jP_{i,j}}{d}\binom{\sum_{i,j}(i-1-2j)P_{i,j}}{c}\prod_{i,j}\binom{i-1}{j,j}^{P_{i,j}}.
\end{equation*}
\noindent The equation above is true for $h =  \sum_{i,j}jP_{i,j}-d$. Otherwise $[B^{c}B'^{d}A_w^{h}{\bf t}^P]\Phi_W^p =0$ and other values of $h$ lead to a zero contribution to the global sum. In a similar fashion, one finds
\begin{align*}
&[W^{e}W'^{f}A_b^{i}{\bf u}^Q]\Phi_B^q =\\
&\hspace{3cm} \frac{q!}{Q!}\left (\frac{1}{2}\right)^{\sum_{i,j}jQ_{i,j}}\binom{\sum_{i,j}jQ_{i,j}}{f}\binom{\sum_{i,j}(i-1-2j)Q_{i,j}}{e}\prod_{i,j}\binom{i-1}{j,j}^{Q_{i,j}},\\
&[B^{q-a-c}B'^{q'-b-d}A_w^{r-q'-g-h}{\bf t'}^{P'}]\Phi_{W'}^{p'} =\\
&\hspace{3cm}\frac{p'!}{P'!}\left (\frac{1}{2}\right)^{\sum_{i,j}jP'_{i,j}-p'}\binom{\sum_{i,j}jP'_{i,j}}{q'-b-d}\binom{\sum_{i,j}(i-2j)P'_{i,j}}{q-a-c}\prod_{i,j}\binom{i-1}{j,j-1}^{P'_{i,j}},\\
&[W^{p-e}{W'}^{p'-f}A_b^{r-p'-i}{\bf u'}^{Q'}]\Phi_{B'}^{q'} =\\
&\hspace{3cm}\frac{q'!}{Q'!}\left (\frac{1}{2}\right)^{\sum_{i,j}jQ'_{i,j}-q'}\binom{\sum_{i,j}jQ'_{i,j}}{p'-f}\binom{\sum_{i,j}(i-2j)Q'_{i,j}}{p-e}\prod_{i,j}\binom{i-1}{j,j-1}^{Q'_{i,j}},\\
&[B^{a}B'^{b}A_w^{g}]\Phi_H = \left (\frac{1}{2}\right)^{j_0}\binom{j_0}{b}\binom{i_0-2j_0}{a}\binom{i_0}{j_0,j_0}.
\end{align*}
\begin{remark}One can check that there is exactly one set of parameters $g,h,i$ that leads to a non zero contribution as $(r-q'-g-h) + g + h = \sum_{i,j}j(P_{i,j}+P'_{i,j}) +j_0 - d -b -(q'-b-d)$ and $(r-p'-i) + i = \sum_{i,j}j(Q_{i,j}+Q'_{i,j})  - f -(p'-f)$.
\end{remark} 
Substituting these expressions in Equation \ref{eq: G}  and summing over the parameters $a,b,c,d,e,f$ leads to the explicit formulation of $G(A)$.\\
The computation can be performed as follows:
\begin{itemize}
\item First as $\sum_{i,j}jQ_{i,j}+\sum_{i,j}jQ'_{i,j} = \sum_{i,j}jP_{i,j}+\sum_{i,j}jP'_{i,j}+j_0 = r$, we have\\
\begin{align*}
G({\bf A}) =&\frac{{p'!^2q'!^2p!q!(r-p')!(r-q')!}}{pp'qq'2^{2r-p'-q'}{\bf A}!}\binom{i_0}{j_0,j_0}\sum_{a,b,c,d,e,f}\Delta'\binom{i_0-2j_0}{a}\binom{\sum_{i,j}(i-2j)P'_{i,j}}{q-a-c}\binom{j_0}{b}\\
&\times\binom{\sum_{i,j}jP_{i,j}}{d}\binom{\sum_{i,j}(i-1-2j)P_{i,j}}{c}\binom{\sum_{i,j}jQ_{i,j}}{f}\binom{\sum_{i,j}(i-1-2j)Q_{i,j}}{e}\\
&\times\binom{\sum_{i,j}jP'_{i,j}}{q'-b-d}\binom{\sum_{i,j}jQ'_{i,j}}{p'-f}\binom{\sum_{i,j}(i-2j)Q'_{i,j}}{p-e}\prod_{i,j}{\binom{i-1}{j,j}}^{(P+Q)_{i,j}}{\binom{i-1}{j,j-1}}^{(P'+Q')_{i,j}}.
\end{align*}
\item Then we sum over $a$ by rewriting the determinant $\Delta'$ as
\begin{align*} \Delta' &= \left |
\begin{array}{cccc}
 0 & 0 & a & b\\
 p & 0 & -c & -d\\
 -e & -f & q & 0\\
 e-p & f-p' & 0 & q'
\end{array}
\right |
= a   \underbrace{\left |
\begin{array}{ccc}
 p & 0  & -d\\
 -e & -f & 0\\
 e-p & f-p' & q'
\end{array}
\right |}_\text{$\Delta''$} 
-b\underbrace{\left |
\begin{array}{ccc}
 p & 0  & -c\\
 -e & -f & q\\
 e-p & f-p' & 0
\end{array}
\right |}_\text{$\Delta'''$}.
\end{align*}
\noindent Thanks to Vandermonde's convolution 
\begin{align*}
\sum_{a}b\Delta'''\binom{i_0-2j_0}{a}\binom{\sum_{i,j}(i-2j)P'_{i,j}}{q-a-c} &= b\Delta'''\binom{i_0-2j_0+\sum_{i,j}(i-2j)P'_{i,j}}{q-c},\\
\sum_{a}a\Delta''\binom{i_0-2j_0}{a}\binom{\sum_{i,j}(i-2j)P'_{i,j}}{q-a-c}&= \sum_{a}\Delta''(i_0-2j_0)\binom{i_0-2j_0-1}{a-1}\binom{\sum_{i,j}(i-2j)P'_{i,j}}{q-a-c},\\
 &= \Delta''(i_0-2j_0)\binom{i_0-2j_0+\sum_{i,j}(i-2j)P'_{i,j}-1}{q-c-1}.
\end{align*}\\
\noindent We have
\begin{align*}
&G({\bf A}) =\frac{{p'!^2q'!^2p!q!(r-p')!(r-q')!}}{pp'qq'2^{2r-p'-q'}{\bf A}!}\binom{i_0}{j_0,j_0}\prod_{i,j}{\binom{i-1}{j,j}}^{(P+Q)_{i,j}}{\binom{i-1}{j,j-1}}^{(P'+Q')_{i,j}} \\
&\times\sum_{b,c,d,e,f}\left[\Delta''(i_0-2j_0)\binom{i_0-2j_0+\sum_{i,j}(i-2j)P'_{i,j}-1}{q-c-1}-b\Delta'''\binom{i_0-2j_0+\sum_{i,j}(i-2j)P'_{i,j}}{q-c}\right]\\
&\times\binom{j_0}{b}\binom{\sum_{i,j}jP_{i,j}}{d}\binom{\sum_{i,j}(i-1-2j)P_{i,j}}{c}\binom{\sum_{i,j}jQ_{i,j}}{f}\binom{\sum_{i,j}(i-1-2j)Q_{i,j}}{e}\\
&\times\binom{\sum_{i,j}jP'_{i,j}}{q'-b-d}\binom{\sum_{i,j}jQ'_{i,j}}{p'-f}\binom{\sum_{i,j}(i-2j)Q'_{i,j}}{p-e}.
\end{align*}
\item We proceed with the summation over $b$:
\begin{align*}
&G({\bf A}) =\frac{{p'!^2q'!^2p!q!(r-p')!(r-q')!}}{pp'qq'2^{2r-p'-q'}{\bf A}!}\binom{i_0}{j_0,j_0}\prod_{i,j}{\binom{i-1}{j,j}}^{(P+Q)_{i,j}}{\binom{i-1}{j,j-1}}^{(P'+Q')_{i,j}} \\
&\times\sum_{c,d,e,f}\left[\Delta''(i_0-2j_0)\binom{i_0-2j_0+\sum_{i,j}(i-2j)P'_{i,j}-1}{q-c-1}\binom{\sum_{i,j}jP'_{i,j}+j_0}{q'-d}-\right.\\
&\hspace{5cm}\left. j_0\Delta'''\binom{i_0-2j_0+\sum_{i,j}(i-2j)P'_{i,j}}{q-c}\binom{\sum_{i,j}jP'_{i,j}+j_0-1}{q'-d-1}\right]\\
&\times\binom{\sum_{i,j}jP_{i,j}}{d}\binom{\sum_{i,j}(i-1-2j)P_{i,j}}{c}\binom{\sum_{i,j}jQ_{i,j}}{f}\binom{\sum_{i,j}(i-1-2j)Q_{i,j}}{e}\\
&\times\binom{\sum_{i,j}jQ'_{i,j}}{p'-f}\binom{\sum_{i,j}(i-2j)Q'_{i,j}}{p-e}.
\end{align*}
\item Then notice that $\Delta''' = -c\Delta''''+pq(p'-f)$ with $\Delta'''' = \left |
\begin{array}{cc}
 -e & -f \\
-p & -p'
\end{array}
\right |$ and $i_0-2j_0+\sum_{i,j}(i-2j)P'_{i,j}+\sum_{i,j}(i-1-2j)P_{i,j} = n-2r-p$. Summing over $c$ gives:
\begin{align*}
&G({\bf A}) =\frac{{p'!^2q'!^2p!q!(r-p')!(r-q')!}}{pp'qq'2^{2r-p'-q'}{\bf A}!}\binom{i_0}{j_0,j_0}\prod_{i,j}{\binom{i-1}{j,j}}^{(P+Q)_{i,j}}{\binom{i-1}{j,j-1}}^{(P'+Q')_{i,j}} \\
&\times\sum_{d,e,f}\left[\Delta''(i_0-2j_0)\binom{n-2r-p-1}{q-1}\binom{\sum_{i,j}jP'_{i,j}+j_0}{q'-d}-\right.\\
&\hspace{3cm}\left. pq(p'-f)j_0\binom{n-2r-p}{q}\binom{\sum_{i,j}jP'_{i,j}+j_0-1}{q'-d-1}-\right.\\
&\hspace{3cm}\left.\Delta''''j_0\sum_{i,j}{(i-1-2j)P_{i,j}}\binom{n-2r-p-1}{q-1}\binom{\sum_{i,j}jP'_{i,j}+j_0-1}{q'-d-1}\right]\\
&\times\binom{\sum_{i,j}jP_{i,j}}{d}\binom{\sum_{i,j}jQ_{i,j}}{f}\binom{\sum_{i,j}(i-1-2j)Q_{i,j}}{e}\binom{\sum_{i,j}jQ'_{i,j}}{p'-f}\binom{\sum_{i,j}(i-2j)Q'_{i,j}}{p-e}.
\end{align*}
\item As $\Delta'' = -d\Delta''''-fq'p$, summing over $d$ yields
\begin{align*}
&G({\bf A}) =\frac{{p'!^2q'!^2p!q!(r-p')!(r-q')!}}{pp'qq'2^{2r-p'-q'}{\bf A}!}\binom{i_0}{j_0,j_0}\prod_{i,j}{\binom{i-1}{j,j}}^{(P+Q)_{i,j}}{\binom{i-1}{j,j-1}}^{(P'+Q')_{i,j}}\\
&\times\sum_{e,f}-\left[\Delta''''\left(\sum_{i,j}(i_0j-j_0(i-1))P_{i,j}\right)+frp(i_0-2j_0)+(p'-f)p(n-2r-p)\right]\\
&\times\binom{r-1}{q'-1}\binom{n-2r-p-1}{q-1}\binom{\sum_{i,j}jQ_{i,j}}{f}\binom{\sum_{i,j}(i-1-2j)Q_{i,j}}{e}\\
&\times\binom{\sum_{i,j}jQ'_{i,j}}{p'-f}\binom{\sum_{i,j}(i-2j)Q'_{i,j}}{p-e}.
\end{align*}
\item One proceeds in a similar fashion to sum over $e$ and $f$:
\begin{align*}
G({\bf A}) =&\frac{{p'!^2q'!^2p!q!(r-p')!(r-q')!}}{pp'qq'2^{2r-p'-q'}{\bf A}!}\prod_{i,j}{\binom{i-1}{j,j}}^{(P+Q)_{i,j}}{\binom{i-1}{j,j-1}}^{(P'+Q')_{i,j}} \\
&\times pr^2\mathcal{I}({\bf A})\binom{r-1}{q'-1}\binom{r-1}{p'-1}\binom{n-2r-p-1}{q-1}\binom{n-2r-q}{p}.
\end{align*}
\noindent Finally,
\begin{align*}
G({\bf A}) =\frac{p'!q'!r!^2(n-1-2r-p)!(n-2r-q)!\mathcal{I}({\bf A})}{2^{2r-p'-q'}(n-p-q-2r)!^2{\bf A}!}\prod_{i,j}{\binom{i-1}{j,j}}^{(P+Q)_{i,j}}{\binom{i-1}{j,j-1}}^{(P'+Q')_{i,j}}.
\end{align*}
\end{itemize}
By replacing $G({\bf A})$ with the expression above in Lemma \ref{lem: f2g}, one gets the desired result:
\begin{equation*}
F({\bf A}) =\frac{\mathcal{I}({\bf A})}{{\bf A}!}\frac{r!^2(n-1-2r-p)!(n-2r-q)!}{2^{2r-p'-q'}(n-p-q-2r)!}\prod_{i,j}{\binom{i-1}{j,j}}^{(P+Q)_{i,j}}{\binom{i-1}{j,j-1}}^{(P'+Q')_{i,j}}.
\end{equation*}
\subsection{Demonstration of Equation \ref{eq:nodegrees} (sketch)}
The computation of $F_{p,p',q,q',r}$ can be performed in a similar fashion. We define the number $\widetilde{G}_{p,p',q,q',r}$ of modified forests as in Section \ref{subsec: andeg} with one major difference in point (i) that we replace by (i'):
\begin{itemize}
\item[(i')] {\bf There is no thorns} connected to the vertices.
\end{itemize}
As an immediate result the two quantities are linked through the relation

$$F_{p,p',q,q',r} = \binom{n}{n+1-p-q-2r}\binom{n-1}{n+1-p-q-2r}\frac{(n+1-p-q-2r)!}{p'!q'!}\widetilde{G}_{p,p',q,q',r},$$

\noindent where the first binomial coefficient counts the number of way of positioning the thorns around the white vertices and the second one, around the black vertices. 

 \noindent We need to compute the generating function 
\begin{equation}
H = \sum_{p,p',q,q',r} \widetilde{G}_{p,p',q,q',r} x^{p}y^{q}\frac{x'^{p'}}{p'!}\frac{y'^{q'}}{q'!}\frac{{f_1}^{r-q'}}{(r-q')!}\frac{{f_2}^{r-p'}}{(r-p')!}
\end{equation}
Define the generating functions $W$ (for internal white vertices {\bf and} the root of the seed tree), $W'$, $B$, $B'$, $A_b$ and $A_w$ (consistent definition with respect to the previous subsection). In this case $W = H$. The new relations between the generating functions read

\begin{equation}
W = x\sum_{i,j,k\geq 0}\binom{i}{2j}B^{i-2j}(2j-1)!!\frac{B'^k}{k!}\frac{A_w^{j-k}}{(j-k)!},
\end{equation}
that simplifies as $\sum_{i\geq 0}\binom{i}{2j}B^{i-2j} = (1-B)^{-2j-1}$:
\begin{equation}
W = \frac{x}{1-B}\sum_{j\geq 0}\binom{2j}{j}\left (\frac{B'+A_w}{2(1-B)^2}\right)^j.
\end{equation}
Finally,
\begin{equation}
W = \frac{x}{1-B}\frac{1}{\sqrt{1-4\left (\frac{B'+A_w}{2(1-B)^2}\right)}}.
\end{equation}
Further we get
\begin{equation}
B = \frac{y}{1-W}\frac{1}{\sqrt{1-4\left (\frac{W'+A_b}{2(1-W)^2}\right)}},
\end{equation}
\begin{equation}
W' = x'\frac{1-\sqrt{1-4\left (\frac{B'+A_w}{2(1-B)^2}\right)}}{\sqrt{1-4\left (\frac{B'+A_w}{2(1-B)^2}\right)}},
\end{equation}
\begin{equation}
B' = y'\frac{1-\sqrt{1-4\left (\frac{W'+A_b}{2(1-W)^2}\right)}}{\sqrt{1-4\left (\frac{W'+A_b}{2(1-W)^2}\right)}}.
\end{equation}

Applying the Lagrange formula for implicit functions as in the previous case leads to the desired formula.
\bibliographystyle{abbrvnat}
% use the following instead if you encounter problems 
%\bibliographystyle{alpha}
%\bibliography{sample}
%\label{sec:biblio}

\end{document}